\documentclass[format=acmsmall, review=false]{acmart}
\usepackage{style_file}   

\usepackage{booktabs} 
\usepackage[ruled]{algorithm2e} 

\SetAlFnt{\small}
\SetAlCapFnt{\small}
\SetAlCapNameFnt{\small}
\SetAlCapHSkip{0pt}
\IncMargin{-\parindent}

\usepackage{amsmath}
\usepackage{amsfonts}
\usepackage{amsthm}
\usepackage{graphicx}
\usepackage{caption}
\usepackage{subcaption}
\usepackage{float}
\usepackage{setspace}
\usepackage[bottom]{footmisc}
\usepackage{tipa}
\usepackage{lscape}
\usepackage{url}
\usepackage{mathtools}
\usepackage{color}
\usepackage{makecell}
\usepackage{hyperref}
\hypersetup{colorlinks,allcolors=black}
\usepackage{listings}
\usepackage{natbib}
\usepackage{wrapfig}
\usepackage[title]{appendix}

\newcommand{\xoverbrace}[2][\vphantom{\dfrac{A}{A}}]{\overbrace{#1#2}}

\newtheorem{example}{Example}[section]
\newtheorem{goal}{Research Goal}[section]

\setcitestyle{acmnumeric}

\title[IO to Learn Cost Function in GNEP]{Using Inverse Optimization to Learn Cost Functions in Generalized Nash Games}

\author[Allen, Dickerson, \& Gabriel]{Stephanie Allen, University of Maryland, College Park \\ John P.~Dickerson, University of Maryland, College Park \\
Steven A.~Gabriel, University of Maryland, College Park \& Norwegian University of Science and Technology}

\begin{abstract}
As demonstrated by \citet{ratliff2014social}, inverse optimization can be used to recover the objective function parameters of players in multi-player Nash games.  These games involve the optimization problems of multiple players in which the players can affect each other in their objective functions. In generalized Nash equilibrium problems (GNEPs), a player's set of feasible actions is also impacted by the actions taken by other players in the game; see \citet{facchinei2010generalized} for more background on this problem.  One example of such impact comes in the form of joint/``coupled'' constraints as referenced by \citet{rosen1965existence}, \citet{harker1991generalized}, and \citet{facchinei2007generalized} which involve other players' variables in the constraints of the feasible region.   We extend the framework of \citet{ratliff2014social} to find inverse optimization solutions for the class of GNEPs with joint constraints.  The resulting formulation is then applied to a simulated multi-player transportation problem on a road network.  Also, we provide some theoretical results related to this transportation problem regarding runtime of the extended framework as well as uniqueness and nonuniqueness of solutions to our simulation experiments.  We see that our model recovers parameterizations that produce the same flow patterns as the original parameterizations and that this holds true across multiple networks, different assumptions regarding players' perceived costs, and the majority of restrictive capacity settings and the associated numbers of players.  Code for the project can be found at: \url{https://github.com/sallen7/IO_GNEP} 

\end{abstract}

\begin{document}

\begin{titlepage}

\maketitle

\end{titlepage}

\section{Introduction}\label{sec:intro} 

In traditional optimization problems, a model with exogenous data is given, and the goal is to find a feasible solution that maximizes or minimizes a given objective function.  In \emph{inverse optimization} (IO) problems, a particular solution or set of solutions is instead given as input, and the goal is to find the parameters of the \emph{optimization problem} that resulted in that observed solution or set of solutions~\citep{ahuja2001inverse,bertsimas2015data,chan2019inverse,zhang2011inverse,ghobadi2021inferring,chan2020inverse,tan2020learning}.  Applications of inverse optimization are myriad, including those in medical~\citep{chan2014generalized}, energy~\citep{saez2017short}, and transportation areas~\citep{zhang2018price}.  

As a more specific example of inverse optimization, we may take as input observations of strategic behavior from rational, utility-maximizing players acting in equilibrium in a non-cooperative game and then infer the utility functions being optimized by those players' behaviors, as was done by those such as \citet{ratliff2014social}, \citet{waugh2011computational,waugh2013computational}, and \citet{risanger2020inverse}. More specifically, \citet{ratliff2014social} observe the outputs of a multi-player Nash game in which the players can affect each other's objective functions in their individualized optimization problems and, then, utilize those outputs to parameterize the players' objective functions.  In the present paper, we focus on the context of a multi-player transportation generalized Nash game with a joint constraint in which players travel across a road network with the goal of minimizing travel time, and we utilize their observed behavior to parameterize their underlying cost functions in the game (see Figure~\ref{fig:illustration_of_framework}).  In parameterizing this model, we extend \citet{ratliff2014social}'s framework to problems with ``coupled''/joint constraints as referenced by \citet{rosen1965existence}, \citet{harker1991generalized}, and \citet{facchinei2007generalized} which involve other players' variables in the constraints of the feasible region.  This means we are parameterizing the objective functions of players in a game in which the players can affect not only each other's objective functions but also each other's feasible regions.  These games with joint constraints are a particular instance of \emph{generalized Nash equilibrium problems (GNEP)} \cite{rosen1965existence,harker1991generalized,facchinei2007generalized,facchinei2010generalized,gabriel2012complementarity}.  In the case of the transportation problem, the joint constraint is a capacity constraint that requires the flow across all players on the various arcs to stay below a certain threshold.

\begin{wrapfigure}{r}{0.5\textwidth}
\centering
\includegraphics[width=0.48\textwidth]{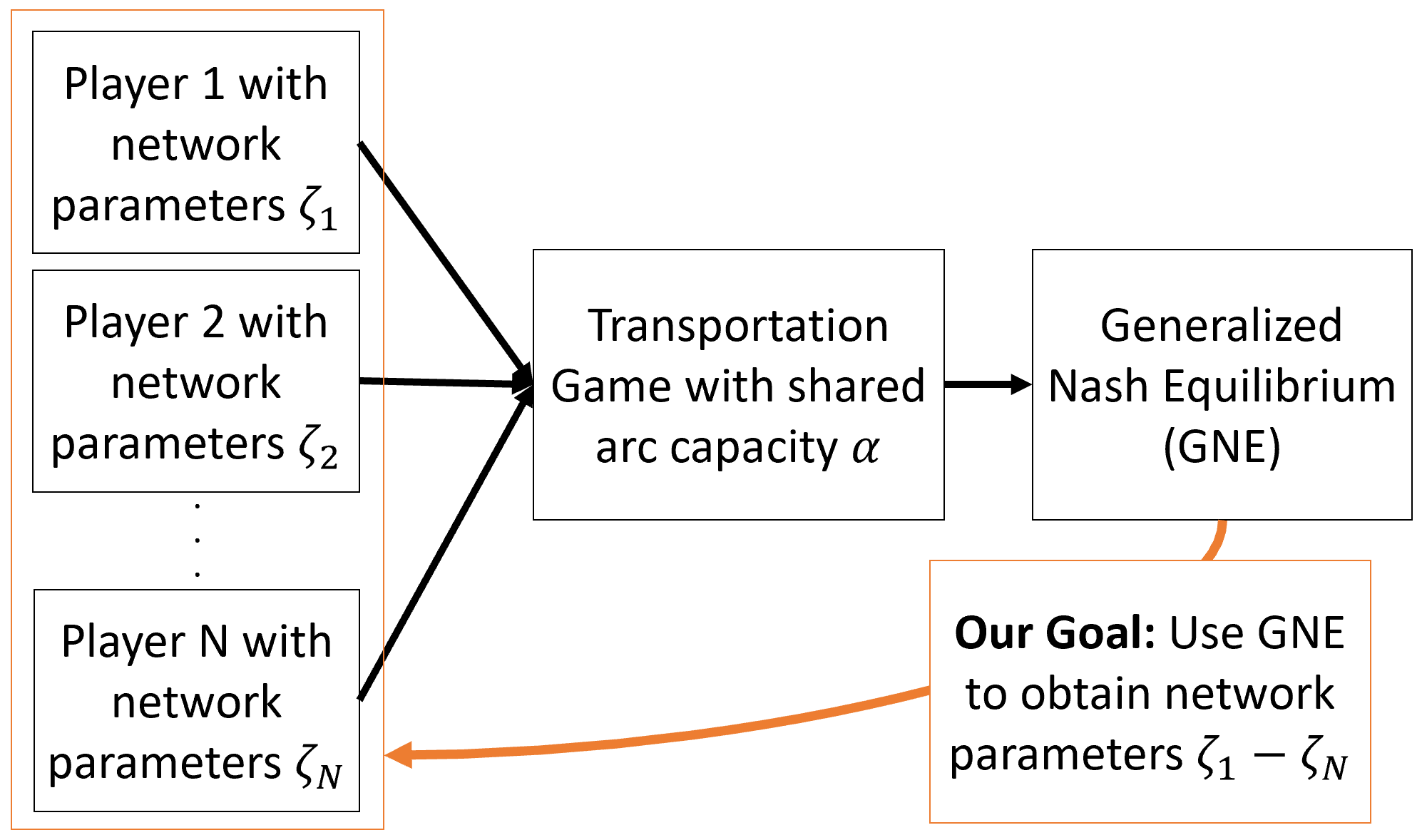}
\caption{
Framework for the Transportation Game
}\label{fig:illustration_of_framework}
\end{wrapfigure}

Although we focus on this transportation problem, we propose that our extension of the \citet{ratliff2014social} framework works beyond this problem, and it will be presented more generally in Section \ref{sec:theory-background}.  We validate its use on small but illustrative transportation networks, showing via simulated experiments with multiple players that we can recover objective function parameters for the players using observed player behavior in GNEPs with joint constraints such that the same flow patterns hold under the recovered parameterizations as under the original cost parameterizations. We also provide some supporting theoretical results regarding a special case of our transportation problem and the existence of unique solutions to this transportation problem.


\subsection{Related Work}\label{sec:intro-rw}

Work due to~\citet{keshavarz2011imputing} on parameterizing the objective function of a convex program has been influential in the inverse optimization community and has been discussed and referenced in multiple other inverse optimization papers such as \citep{ratliff2014social,bertsimas2015data,aswani2018inverse,esfahani2018data}.  \citet{ratliff2014social} apply the ideas of that paper to the context of game theory by making the argument that the KKT conditions \citep{boyd2004convex} and other added constraints parallel different properties of Nash equilibria, as further illustrated by \cite{rosen1965existence} and \cite{ratliff2013characterization}.  \citet{ratliff2014social} use data to parameterize the objective functions of players in an energy game.  This combination of game theory with models/algorithms to learn parameterizations is not new; indeed, there is a subcommunity of the economics and computer science community that seeks to estimate the parameters of the players in their games \citep{waugh2013computational,waugh2011computational,ratliff2014social,peysakhovich2019robust,risanger2020inverse,nekipelov2015econometrics}.  Notably, \citet{nekipelov2015econometrics} reconstructs the incentives of players using data by assuming no-regret learning on the part of the players, which allows for more realistic modeling of players than the stronger Nash assumption of perfect knowledge. \citet{waugh2013computational,waugh2011computational} also encode the requirement of low regret into their inverse optimization model of multi-player interaction in a correlated equilibrium game while maximizing entropy of the mixed-strategy actions of the players.  Finally, \citet{peysakhovich2019robust} point to the importance of having robust methods when estimating players' parameters because our rationality assumptions can be problematic, and our models of the players might be incorrect.

With regard to transportation problems, our appliation of interest in the present work, the methods of inverse optimization and game theory have both been applied before.  The inverse optimization literature on transportation problems has focused on parameterizing cost functions \citep{xu2018network,zhang2017data,nguyen2010application,duin2006some,jain2013inverse}, finding capacity values on a network \citep{chow2014nonlinear}, and also on finding the ``weights'' for components of objective functions in transportation-themed models \citep{chow2012inverse,chow2015activity}.  In reference to game theory and transportation, \citet{elvik2014review} reviews a number of game theory models that deal with such transportation phenomenon as ``speed choice'' and ``merging,'' and \citet{zhang2010review} identify types of games that appear in transportation research which include generalized Nash equilibrium games.  In a recent work, \citet{stein2018noncooperative} solve a multi-player version of the bipartite graph transportation problem as a GNEP, thus demonstrating that generalized Nash problems are a difficult problem type that is still being explored. The merging of game theory/equilibrium problems, the estimation of parameters for these problems, and transportation has been seen in a number of papers~\citep{thai2015multi,bertsimas2015data,zhang2017data,zhang2018price,chow2014nonlinear,thai2017learnability,kuleshov2015inverse,konstantakopoulos2019deep,thai2018imputing,bhaskar2019achieving}. 
In particular, \citet{thai2017learnability,kuleshov2015inverse}, and \citet{bhaskar2019achieving} focus on and/or cover the combination of routing games in transportation with estimating parameters, a category into which our paper more specifically fits.  Other researchers have pursued similar work to ours in the inverse GNEP space \citep{hemmecke2009nash}, taking more of a pure mathematical approach.  Still others have utilized a joint capacity constraint in their models of player movement \cite{xu2018network} in an effort to estimate ``network dual prices'' for the constraint. 

For the papers that closely align with the current paper, we contrast their methodology with ours.  In \citet{thai2017learnability}, the researchers focus on ``minimizing risk'' with regard to estimating the cost functions in a routing game, and the reseachers do not deal with any joint constraints.  By contrast, we focus on KKT/optimality conditions instead of statistical quantities and tackle a different kind of problem in the form of a generalized Nash equilibrium problem.  In \citet{bhaskar2019achieving}, the researchers approach estimating parameters in a routing problem through a ``query'' system in which they: (a) control part of the objective function to obtain a desired flow pattern and (b) control part of the flow to obtain a desired flow pattern.  This contrasts with our approach since we recover a set of parameterizations instead of influencing the flow, and we again work with a generalized Nash game while these researchers do not.  \citet{kuleshov2015inverse} most closely align with our work because the researchers' work can be directly applied to atomic routing games, as opposed to the previous two papers \cite{thai2017learnability,bhaskar2019achieving} that focused upon non-atomic routing games in which individual players are not considered (see \cite{AGT_routing_games} for more details on atomic versus non-atomic routing games).  
While \citet{kuleshov2015inverse} devise a mathematical program to solve the problem of estimating parameters, similar to our approach, they do not deal with a joint constraint in their approach.  \citet{hemmecke2009nash} focus on a theoretical framework for finding cost functions in generalized Nash games, which contrasts with our approach that extends an applied framework using optimality conditions to find cost functions in generalized Nash games with joint constraints.  Finally, \citet{xu2018network} formulate a set of linear programs based on work by \citet{ahuja2001inverse} to handle their problem of finding ``network dual prices'' for their joint capacity constraint and utilize an averaging algorithm to find these prices.  Although \citet{xu2018network} work with a joint constraint, we work with and extend another model from \citet{ratliff2014social} and \citet{keshavarz2011imputing} that has been influential in the literature (as seen in \cite{ratliff2014social,bertsimas2015data,aswani2018inverse,esfahani2018data}). 

\subsection{Our Contribution}\label{sec:intro-contrib}

We contribute to the literature by extending the parameterization framework due to~\citet{ratliff2014social} to the context of GNEP games in which there are joint constraints among the players' actions~\citep{rosen1965existence,harker1991generalized,facchinei2007generalized,facchinei2010generalized,gabriel2012complementarity}.  The approach is motivated in particular by a transportation game in which there is a joint flow constraint for arcs in a road network, due to the fact that roads naturally have capacity limitations.  We differ from the transportation papers that combine equilibrium problems with inverse optimization, such as \citet{bertsimas2015data,zhang2018price}, and the other papers listed above involving estimating parameters, including the non-atomic routing game literature of \citet{thai2017learnability} and \citet{bhaskar2019achieving}, because our method models traffic at the level of the player and includes this joint constraint.  Our approach also differs from \citet{kuleshov2015inverse} who model traffic at the level of the player because of this joint constraint.  What's more, while \citet{hemmecke2009nash} discuss theoretical results with regard to inverse optimization and GNEPs, we pursue an applied framework to recover parameterizations for cost functions.  Finally, although one of \citet{xu2018network}'s models includes a joint constraint, we focus on extending the framework by \citet{ratliff2014social} and \citet{keshavarz2011imputing} as opposed to their use of \citet{ahuja2001inverse}'s framework.  Our proposed approach can take data from simulations (as in this work) or real-world observations (in the case of a real-world policy-maker) involving equilibrium solutions to these problems, and can parameterize the objective functions of these players.  The contribution to the literature is significant because we extend a framework that has been influential in the literature (as seen in \cite{ratliff2014social,bertsimas2015data,aswani2018inverse,esfahani2018data}) to this new class of problems.  

\section{Preliminaries}\label{sec:prelims}

We examine the problem of combining $N$-players' shortest path problems to create a traffic game in which the $N$ players affect each other's costs on the network via interaction terms in their objective functions as well as via a joint capacity constraint in the feasible region.
The problem parallels what is known as an atomic routing game \cite{AGT_routing_games} which features multiple players each with the ability to affect the outcome of the traffic game and which can be represented as a multi-user Nash equilibrium game \cite{cominetti2006network,orda1993competitive}.  However, atomic routing games as presented in \citet{AGT_routing_games}
do not have joint constraints.  We do note that this model was also inspired by \citet{xu2018network} who formulate their problem as a set of \textit{inverse} shortest path problems and who did include a joint capacity constraint in one of their formulations.

For a network with $n$ arcs and $m$ nodes, let $x_i \in \mathbb{R}^n$ to be the flow decision variables for the network arcs for each player $i \in [N]$, which can be fractional.  The positive diagonal matrix $\mathcal{C}_i \in \mathbb{R}^{n\times n}$ represents the coefficients for the interactions between player $i$'s flow and the aggregate flow, capturing the dynamic costs associated with increasing flow from the players.  The vector $\bar{c}_i \in \mathbb{R}^n$ represents the base cost for traveling along each link in the network; this is also called the free flow travel time by various papers~\citep{siri2020progressive,thai2015multi,bertsimas2015data,zhang2017data,zhang2018price,chow2014nonlinear}.  We also take $D \in \{ -1,0,1 \}^{m \times n}$ as the conservation-of-flow matrix which keeps track of the net flow at the nodes~\citep{marcotte2007traffic} and $\alpha \in \mathbb{R}^n$ as the collective capacity for the $N$ players' flow on each arc.  The vector $f_i \in \mathbb{R}^m$ for each player $i$ indicates the starting (in negative units of flow) and ending (in positive units of flow) points in the network of the flow for each player $i \in [N]$.  As an example for $f_i$, if player $i=1$ begins at node 2 and ends at node 12 in a network with 16 nodes, then their $f_1$ vector would have a dimension of $16x1$, and we would place a $-1$ at entry 2 and a 1 at entry 12.   Given all of these variable and parameter definitions, the resulting optimization problem faced by player $i$ is:

\begin{subequations}\label{new_GNE}
\begin{equation}
    \min\limits_{x_i} x_i^T \mathcal{C}_i \left(\sum\limits_{j=1}^N x_j\right) + \bar{c_i}^T x_i  
\end{equation}
\begin{align}
    D x_i = f_i &\text{ (dual variable: }v_i\text{)} \label{new_GNE_conservation} \\
    x_i \geq 0 & \text{ (dual variable: }u_i\text{)} \label{new_GNE_non_neg} \\
    \sum\limits_{j=1}^{N} x_j \leq \alpha & \text{ (dual variable: }\bar{u}\text{)} \label{new_GNE_joint_constraint}
\end{align}
\end{subequations}

\noindent This is a $N$-player traffic game in which the $N$ players affect each other's costs on the network via the objective function, which has an interaction term between player $i$'s flow and all other players' flow and a term representing the base cost for traveling along each arc.  Also, there is a shared constraint in the flows (\ref{new_GNE_joint_constraint}).
This constraint makes this problem a generalized Nash equilibrium problem \cite{gabriel2012complementarity,harker1991generalized,facchinei2007generalized,facchinei2010generalized,rosen1965existence}.  We also have the conservation of flow constraints in (\ref{new_GNE_conservation}) for each player $i$ and the non-negativity constraints in (\ref{new_GNE_non_neg}) for each player $i$.  With this notation defined, we are able to state the goal of the paper. 

\begin{goal}
The goal in this paper is to find the $\mathcal{C}_i$ diagonal matrix and the $\bar{c}_i$ vector associated with this problem for each player $i$. 
\end{goal}

We will examine the case in which these matrices and vectors are equal across all $N$ players and in which they are different across all $N$ players.  The first case represents the situation in which each player faces the same interaction costs and free flow travel times as all other players which is a common assumption in aggregate models such as \cite{bertsimas2015data,thai2015multi,thai2018imputing}.  The second case represents a view in which the players do not have a consensus view of the interaction costs or the free flow travel times, as proposed by \cite{xu2018network,chow2012inverse,chow2015activity}.    

In order to obtain these matrix and vector parameters, we extend the parameterization formulation proposed by \citet{ratliff2014social} to GNEPs with joint constraints.  \citet{ratliff2014social} call their formulation a residual model after the similar \citet{keshavarz2011imputing} model. Our extended residual model is defined as follows for the $N$ players~\citep[see, e.g.,][]{ratliff2014social,konstantakopoulos2017robust,keshavarz2011imputing}:
\begin{subequations}\label{empirical_residual_model_1}
\begin{small}
\begin{equation*}
    \min\limits_{\mathcal{C}_i,\bar{c}_i,v_i^k,u_i^k,\bar{u}^k} \sum\limits_{k} \left( \xoverbrace{\sum\limits_{i=1}^N ||\mathcal{C}_i \left(2 x_i^k + \sum\limits_{j\neq i}x_{j}^k\right) + \bar{c}_i + D^T v_i^k - u_i^k + \bar{u}^k||_1}^{\text{stationarity}} + \xoverbrace{\sum\limits_{i=1}^{N} ||(x_i^k)^T u_i^k ||_1}^{\text{complementarity \#1}} + \xoverbrace{||(\alpha-\sum\limits_{j}x_{j}^k)^T(\bar{u}^k)||_1}^{\text{complementarity \#2}} \right)
\end{equation*}
\end{small}
\begin{equation}\label{empirical_residual_model_1:C_cal_bounds}
    L_1 \leq diag(\mathcal{C}_i) \leq U_1 \ \forall i
\end{equation}
\begin{equation}\label{empirical_residual_model_1:c_bar_bounds}
    L_2 \leq \bar{c}_i \leq U_2 \ \forall i
\end{equation}
\begin{equation}\label{empirical_residual_model_1:dual_var_bounds}
    \bar{u}^k \geq 0\ \forall k\ u_i^k \geq 0\ \forall i,k
\end{equation}
\end{subequations}

At a high level, our objective function naturally decomposes into three parts (labeled above as ``stationarity,'' ``complementarity \#1,'' and ``complementarity \#2''); we now describe the intuition behind each part.
\begin{itemize}
    \item The \textbf{stationarity} component is a norm corresponding to the stationarity conditions of the optimization problem for each player $i$ and each data piece $k$, with the data pieces representing the flow patterns we observe between the origin-destination pairs on the network.  For this norm, the $x_i^k$ are data of the flow patterns for each player $i$ and each origin-destination pair $k$, and our variables are the diagonal of the $\mathcal{C}_i\ \forall i$, the vector $\bar{c}_i\ \forall i$, $u_i^k$ as the dual variables for constraint \ref{new_GNE_non_neg} for all $i,k$, $v_i^k$ as the dual variables for constraint \ref{new_GNE_conservation} for all $i,k$, and $\bar{u}^k$ as the dual variables for constraint \ref{new_GNE_joint_constraint} for each data piece/origin-destination pair $k$.  An origin-destination (OD) pair indicates the beginning (origin) and ending (destination) points of a player's path. We have one set of $\bar{u}^k$ dual variables for each $k$ that are shared across the players, which will be further explained in Section \ref{sec:theory-background}. 
    \item The \textbf{complementarity \#1} norm refers to the complementarity conditions for the nonnegativity constraint across the data and the players
    \item The \textbf{complementarity \#2} norm  refers to the complementarity conditions that pair with the joint constraint for each data piece.
\end{itemize}
In conjunction, these three components of the objective allow us to minimize the error across the relevant KKT conditions for all players and all data.  The constraints on this objective include requiring that the diagonal of the $\mathcal{C}_i$ matrices remain between two bounds $L_1\in \mathbb{R}^n$ and $U_1\in \mathbb{R}^n$ as captured in (\ref{empirical_residual_model_1:C_cal_bounds}) and that the vectors $\bar{c}_i$ remain between two bounds $L_2\in \mathbb{R}^n$ and $U_2\in \mathbb{R}^n$ as captured in (\ref{empirical_residual_model_1:c_bar_bounds}).  The use of upper and lower bounds for the $\mathcal{C}_i$ and $\bar{c}_i$ was inspired by papers such as \cite{keshavarz2011imputing,ratliff2014social,konstantakopoulos2017robust,chan2019inverse,bertsimas2015data} which discuss utilizing ``prior information'' \cite{keshavarz2011imputing} pertaining to the parameters to be estimated to aid the mathematical program in finding viable parameterizations.  We remember that the diagonals of $\mathcal{C}_i$ represent the interaction costs for each player $i$ and the vectors $\bar{c}_i$ represent the free flow travel cost for each player $i$ so, by placing bounds on these variables, their estimated values will be on the same order of magnitude as the original values that generated the simulation data.\footnote{We do want to ensure that a set of parameterizations results in a convex Nash game since we are dealing with a set of minimization problems across players, and convex games result in Nash equilibrium solutions \cite{rosen1965existence,ratliff2014social,orda1993competitive}.  Due to the fact that our minimization problems for all of the players are linear in each player's variables, we do not require any extra constraints \cite{rosen1965existence,ratliff2014social,orda1993competitive}, but this might be required for other applications as demonstrated by \cite{konstantakopoulos2017robust}.}  The final constraint (\ref{empirical_residual_model_1:dual_var_bounds}) on the objective requires that the dual variables $\bar{u}^k$ and $u_i^k$ are non-negative because they correspond with inequality constraints.

One important point about this model is that it can be solved in polynomial time.  From \citet{boyd2004convex}, we know we can convert the L-1 norm components into linear constraints, with positive variables that bound the constraints from above and below and that replace the L-1 terms in the objective function.  This leads to a linear program, which we also know from \cite{boyd2004convex} can be solved in polynomial time via the interior point method.  Furthermore, the growth in the number of variables as a function of the inputs also grows as a polynomial function, the details of which can be seen in Appendix \ref{appendix:variable_number}.  Overall, we summarize all of this in the following lemma:
\begin{lemma}\label{complexity_lemma}
Problem (\ref{empirical_residual_model_1}) can be solved in polynomial time, and the number of variables grows as a polynomial function of the inputs.
\end{lemma}

Having established the complexity of our inverse optimization model, we can now discuss an example of our framework.

\begin{example}\label{example_of_framework}
We randomly generate a parameterization for the arcs on a 4x4 grid which consists of 48 total arcs.  There are 48 arcs on this grid because there are 4 sets of 3 edges moving down the grid and then the same pattern moving across the grid; therefore, we multiply $(4)(3)$ by 2 and, then, multiply it by 2 again because we assume two arcs going in opposite directions per edge.  We generate this parameterization by uniformly choosing 48 numbers between 2 and 10 for the diagonal of $\mathcal{C}_i$ such that $\mathcal{C}_1 = ... = \mathcal{C}_N = \mathcal{C}$ and uniformly choosing 48 numbers between 2 and 10 for the vector $\bar{c}_i$ such that $\bar{c}_1 = ... = \bar{c}_N = \bar{c}$.  We set the number of players as $N=10$ and then generate the flows for these parameterizations using all 240 origin-destination (OD) pairs - a number obtained by calculating $\binom{16}{2}$, since there are 16 nodes in a 4x4 grid, and then multiplying by 2 to account for our assumption that the arcs from origin $a$ to destination $b$ might be different from the arc from origin $b$ to destination $a$.  For each OD pair, all of the players each start this one unit of flow at the origin and flow this one unit of flow to the destination, meaning that for $N=10$ players there are 10 units of flow on the network. We then input the resulting flows under these origin-destination pairs into the residual model above (\ref{empirical_residual_model_1}) to obtain the parameterizations for $\mathcal{C}$ and $\bar{c}$.  
\end{example}

With regard to the results of example \ref{example_of_framework}, the original parameterizations and the parameterizations obtained by the inverse optimization residual model do not align.  Indeed, the Frobenius norm difference between the two $\mathcal{C}$ matrices is 2.4777, and the 2-norm difference between the two parameterizations for the $\bar{c}$ vector is 22.4281.  However, if we instead use an error metric that compares the flows under the original parameterizations versus the flows under the IO parameterizations, we see promising results.  Indeed, when we compare the flows among the 10 players for all 240 origin destination pairs across all 48 arcs on the network, we find that the difference between the flow patterns using a Frobenius norm metric is 8.1369e-06.  We utilize the Frobenius norm because it allows us to calculate essentially a 2-norm between the two flow matrices, one resulting from the original costs and one resulting from the IO costs.  Consequently, as indicated by the low Frobenius norm error, the IO solution leads to the same type of flow pattern observed under the original parameterization for the set of OD pair configurations considered, thus making the parameters obtained by the IO residual model a reasonable solution to the inverse optimization problem.  

We have now explained the framework at a basic level.  Next,  Section~\ref{sec:theory-background} discusses the connection between the residual model \citep{keshavarz2011imputing,ratliff2014social} and the generalized Nash equilibrium problem (GNEP) along with some other theoretical desiderata and preliminaries. Section~\ref{sec:numerical-setup} presents the explicit models used and the experimental framework.  Section~\ref{sec:experiments} presents our experimental results, and Section~\ref{sec:conclusions} concludes with a brief discussion and open problems.

\section{Theoretical Background and Connections}\label{sec:theory-background}

In this section, we describe our simulation approach, form the residual model more abstractly, and then make the connection between the simulation conditions and the residual model.  In Section \ref{sec:theory-QVI}, we discuss the connection between the subclass of Generalized Nash Equilibrium Problems and stacked KKT conditions; we use these stacked KKT conditions as our simulation model and thus our data generation procedure.  In Section \ref{sec:theory-residual}, we explain the residual model utilized by \citet{ratliff2014social}.  Finally, in Section \ref{sec:theory-gnep}, we propose and prove a collorary stating that we can use the stacked KKT conditions in the residual model to recover the parameterizations for the objective functions of the players in a GNEP.


\subsection{Theoretical Considerations for Generating the Simulation Data}\label{sec:theory-QVI}

Using the notation from \citet{facchinei2010generalized}, for each player $v$ in a generalized Nash problem, the goal is to solve the following optimization problem, with $N$ as the number of players, $n_v$ as the number of variables for player $v$, $x^v \in \mathbb{R}^{n_v}$ as the variables for player $v$, $\mathbf{x}^{-v}$ as containing all other players' variables, $X_v(\mathbf{x}^{-v}) \subseteq \mathbb{R}^{n_v}$ as the ``feasible set'' parameterized by the other players' variables, $\theta:\mathbb{R}^{\bar{n}} \rightarrow \mathbb{R}^{N}$ with $\bar{n} = \sum_v n_v$, and $\theta_v:\mathbb{R}^{\bar{n}} \rightarrow \mathbb{R}$ (with the recognition that the last function has $\mathbf{x}^{-v}$ as parameters in its functional form once we move to the GNEP problem)
\begin{subequations}\label{GNE_general_form}
\begin{equation}
    \min\limits_{x^v} \theta_v(x^v,\mathbf{x}^{-v})
\end{equation}
\begin{equation}
    x^v \in X_v(\mathbf{x}^{-v})
\end{equation}
\end{subequations}

\citet{facchinei2010generalized} and \citet{gabriel2012complementarity} convey that a solution to this problem is an $\mathbf{x}$ such that no player can minimize their objective further, fixing the other players' variables.  We can form a problem to find this solution (known as a generalized Nash equilibrium) as long as we have the Convexity Assumption stated in \cite{facchinei2010generalized}:

\vspace{1em}

\noindent \textbf{Convexity Assumption~\cite{facchinei2010generalized}.} For every player $v$ and every $\mathbf{x}^{-v}$, the objective function $\theta(\cdot,\mathbf{x}^{-v})$ is convex and the set $X_v(\mathbf{x}^{-v})$ is closed and convex.

\vspace{1em}

\noindent With this Convexity Assumption, the following theorem holds in which the solution of a quasi-variational inequality QVI$(\mathbf{X}(\mathbf{x}),\mathbf{F}(\mathbf{x}))$ problem is defined as $\mathbf{x}^* \in \mathbf{X}(\mathbf{x}^*)\subseteq \mathbb{R}^{\bar{n}}$ such that 

\begin{equation}
    \mathbf{F}(\mathbf{x}^*)^T (\mathbf{y} - \mathbf{x}^*) \geq 0\ \forall \mathbf{y} \in \mathbf{X}(\mathbf{x}^*)
\end{equation}

\begin{theorem}[Theorem 3.3 of \cite{facchinei2010generalized}] Let a GNEP be given, satisfying the Convexity Assumption, and suppose further that the $\theta_v$ functions are $C^1$ for all $v$.  Then a point $\mathbf{x}$ is a generalized Nash equilibrium if and only if it is a solution of the quasi-variational inequality QVI$(\mathbf{X}(\mathbf{x}),\mathbf{F}(\mathbf{x}))$, where
\begin{equation}\label{specific_QVI_formulation}
    \mathbf{X}(\mathbf{x}) = \prod\limits_{v=1}^{N} X_v(\mathbf{x}^{-v}),\ \ \mathbf{F}(\mathbf{x}) = \left(\nabla_{x^v} \theta_v(\mathbf{x}) \right)_{v=1}^N
\end{equation}

\end{theorem}

\noindent This QVI representation of the GNEP can be further simplified to a related variational inequality (VI) under certain conditions on the $X_v(\mathbf{x}^{-v})$.  For a variational inequality, the goal is to find $\mathbf{x}^* \in \mathbf{X} \subseteq \mathbb{R}^{\bar{n}}$, with $\mathbf{X}$ not parameterized by the variables, such that
\begin{equation}\label{VI_formulation}
    \mathbf{F}(\mathbf{x}^*)^T (\mathbf{y} - \mathbf{x}^*) \geq 0,\ \forall \mathbf{y} \in \mathbf{X}
\end{equation}

\noindent One way to transform the QVI into a VI is via joint convexity of the GNEP, a property that applies if the following definition holds pertaining to the $X_v(\mathbf{x}^{-v})$ sets:

\begin{definition}[Definition 3.6 of \cite{facchinei2010generalized}] Let a GNEP be given, satisfying the convexity assumption.  We say that this GNEP is jointly convex if for some closed, convex $\mathbf{X} \subseteq \mathbb{R}^{\bar{n}}$ and all $v=1,...,N$, we have 
\begin{equation}
    X_v(\mathbf{x}^{-v}) = \{ x^v \in \mathbb{R}^{n_v} : (x^v,\mathbf{x}^{-v}) \in \mathbf{X} \}
\end{equation}

\end{definition}

\noindent This definition means that the $x^v$ and $\mathbf{x}^{-v}$ are all part of the same set $\mathbf{X}$.  This means that we can combine the constraints from each set $X_v(\mathbf{x}^{-v})$ into one set $\mathbf{X}$, which also implies that we will have one copy of any constraints known as joint constraints that depend upon all the players' variables in the set $\mathbf{X}$.  Furthermore, by having one copy of the joint constraints, this equates the dual variables for this constraint across the players \cite{rosen1965existence,harker1991generalized,facchinei2007generalized,facchinei2010generalized,gabriel2012complementarity}.  As a result, we can utilize the following theorem:  

\begin{theorem}[Theorem 3.9 of \cite{facchinei2010generalized}] Let a jointly convex GNEP be given with $C^1$ functions $\theta_v$.  Then every solution of the $VI(\mathbf{X},\mathbf{F})$ (where $\mathbf{X}$ is the set in the definition of joint convexity and, as usual, $\mathbf{F}(\mathbf{x}) = \left(\nabla_{x^v} \theta_v(\mathbf{x}) \right)_{v=1}^N$), is also a solution of the GNEP.
\end{theorem}

\noindent This theorem states that we can utilize the VI with the combined $\mathbf{X}$ set instead of the QVI with $\mathbf{X}(\mathbf{x})$ to obtain a solution to the original generalized Nash problem \cite{rosen1965existence,harker1991generalized,facchinei2007generalized,facchinei2010generalized,gabriel2012complementarity,facchinei2007finite}.  For our transportation problem, the $\mathbf{X}$ set consists of the constraints of all of the players plus one copy of the joint constraints:
\begin{equation}\label{X_set_for_GNE}
    \mathbf{X} = \left\{\mathbf{x} : Dx_v = f_v\ \forall v, x_v \geq 0\ \forall v, \sum\limits_{v=1}^{N} x_v \leq \alpha \right\}
\end{equation}

We know there exists a solution to this VI based upon the following theorem from \citet{harker1990finite} who cite \cite{eaves1971basic,hartman1966some}:

\begin{theorem}[Theorem 3.1 of \cite{harker1990finite}
] Let $\textbf{X}$ be a nonempty, compact and convex subset of $\mathbb{R}^{\bar{n}}$ and let $\mathbf{F}$ be a continuous mapping from $\mathbf{X}$ into $\mathbb{R}^{\bar{n}}$.  Then there exists a solution to the problem VI$(\mathbf{X},\mathbf{F})$.
\end{theorem}

Due to the fact that our $\mathbf{X}$ is nonempty, compact, and convex as a result of the linear constraints, the lower bound provided by the non-negativity constraints, and the upper bound provided by the capacity constraint and due to the fact that our function $\mathbf{F}$ is continuous, we know that a solution exists for our VI.  Furthermore, according to \citet{harker1991generalized} and \citet{gabriel2012complementarity}, we can obtain a unique solution to this VI if the set $\mathbf{X}$ is compact and convex and if $\mathbf{F}$ defined by (\ref{VI_formulation}) is strictly monotone, which is defined in the following definition:

\begin{definition}[Strict Monotonicity (\cite{harker1990finite,ortega2000iterative,gabriel2012complementarity})] 
A function $\mathbf{F}$ is strictly monotone if
\begin{equation}
(\mathbf{F}(\mathbf{x}) - \mathbf{F}(\mathbf{y}))^T(\mathbf{x}-\mathbf{y}) > 0,\ \forall \mathbf{x},\mathbf{y} \in \mathbf{X},\ \mathbf{x} \neq \mathbf{y}
\end{equation}
\end{definition}

We achieve both criteria with our traffic game, as defined and discussed in Section~\ref{sec:prelims}, because the equality constraints are linear, the non-joint inequality constraints are convex, the joint constraints are convex \citep{harker1991generalized,facchinei2010generalized}, the constraints create bounds on the feasible region, and the $\mathbf{F}$ defined in (\ref{specific_QVI_formulation}) composed of the gradients of the objective functions of the players is strictly monotone \citep{gabriel2012complementarity,harker1991generalized}.  Strict monotonicity of our $\mathbf{F}$ defined in (\ref{specific_QVI_formulation}) is implied by showing it is strongly monotone \citep{harker1990finite,ortega2000iterative,gabriel2012complementarity}, a property captured by the following definition:  
\begin{definition}[Strong Monotonicity (\cite{harker1990finite,ortega2000iterative,gabriel2012complementarity})]
A function $\mathbf{F}$ is strongly monotone if for some $\alpha > 0$
\begin{equation}
(\mathbf{F}(\mathbf{x}) - \mathbf{F}(\mathbf{y}))^T(\mathbf{x}-\mathbf{y}) \geq \alpha ||\mathbf{x} - \mathbf{y} ||_2^2,\ \forall \mathbf{x},\mathbf{y} \in \mathbf{X},\ \mathbf{x} \neq \mathbf{y}
\end{equation}
\end{definition}

We recognize that a similar proposition regarding players that all have the same cost functions and same origin-destination pairs appeared in \citet{orda1993competitive} as described by \citet{cominetti2006network}, but our proof is specialized to this problem.

\begin{lemma}\label{thm:strong-monotonicity}
The $\mathbf{F}$ defined by (\ref{specific_QVI_formulation}) and created by the gradients of the objective functions defined in (\ref{new_GNE}) is strongly monotone if we assume $\mathcal{C}_1 = \mathcal{C}_2 = ... = \mathcal{C}_N$.  
\end{lemma}

See Appendix \ref{appendix:strong-monotonicity-proof} for the proof.  Note though that, just because the VI form for the GNEP has one solution (due to the strong monotonicity property), this does not mean that the original generalized Nash problem without the assumption of equal multipliers for the joint constraints would not have other solutions (see \citet{facchinei2010generalized} for a simple example in Section 1).  
Furthermore, if the $\mathcal{C}_i$ matrices are not all equal, then we are not guaranteed to have a strongly monotone $\mathbf{F}$.

\begin{lemma}\label{lemma:not-all-C-equal}
The $\mathbf{F}$ defined by (\ref{specific_QVI_formulation}) and created by the gradients of the objective functions defined in (\ref{new_GNE}) is not in general strongly monotone if we assume the $\mathcal{C}_i$ are not all equal.
\end{lemma}
See Appendix \ref{appendix_lemma:not-all-C-equal} for the counterexample. Finally, with the VI formulation, we can then use a well-known theorem from \citet{facchinei2007finite} to convert the VI into a mixed complementarity problem.  \citet{facchinei2007finite} define the $\mathbf{X} = \{\mathbf{x}\in \mathbb{R}^{\hat{n}} : A\mathbf{x} \leq b, G\mathbf{x} = d\}$ such that $A \in \mathbb{R}^{\hat{m} \times \hat{n}}, G \in \mathbb{R}^{\hat{l}\times \hat{n}}, b\in \mathbb{R}^{\hat{m}}, d \in \mathbb{R}^{\hat{l}}$.  They have the following proposition:

\begin{proposition}[Proposition 1.2.1 from \cite{facchinei2007finite}]
Let $\mathbf{X}$ be defined as above. A vector $\mathbf{x}$ solves the VI$(\mathbf{X},\mathbf{F})$ if and only if these exists $\lambda \in \mathbb{R}^{\hat{m}}$ and $\mu \in \mathbb{R}^{\hat{l}}$ such that
\begin{subequations}
\begin{equation}
    0 = \mathbf{F}(\mathbf{x}) + G^T \mu + A^T \lambda
\end{equation}
\begin{equation}
    0 = d - G\mathbf{x}
\end{equation}
\begin{equation}
    0 \leq \lambda \ \bot \ b - A\mathbf{x} \geq 0
\end{equation}
\end{subequations}

\end{proposition}

\noindent This proposition states that, if we find a solution to the KKT system, we will solve the VI defined by $\mathbf{X}$ and $\mathbf{F}$, so long as the $\mathbf{X}$ is a set of linear inequality and linear equality constraints.  Therefore, since we have linear equality and linear inequality constraints in our $\mathbf{X}$, we can take the $\mathbf{X}$ defined in (\ref{X_set_for_GNE}) and write the following system for all players $i$ and for a given origin destination pair encoded in $f_i\ \forall i$: 

\begin{subequations}\label{mixed_complementarity_GNEP}
\begin{equation}\label{MCP_stationarity}
    \mathcal{C}_i \left(2x_i + \sum\limits_{j\neq i} x_j\right) + \bar{c}_i + D^T v_i - u_i + \bar{u} = 0\ \forall i = 1,...,N
\end{equation}
\begin{equation}\label{MCP_nonneg}
    0 \leq x_i \ \bot \ u_i \geq 0 \ \forall i = 1,...,N
\end{equation}
\begin{equation}
    0 \leq \alpha-\left(\sum\limits_j x_j\right) \ \bot \ \bar{u} \geq 0
\end{equation}
\begin{equation}
    D x_i = f_i,\ v_i \text{ free}\ \forall i = 1,...,N
\end{equation}
\end{subequations}

\noindent with dual variables of $u_i\ \forall i, \bar{u},$ and $v_i\ \forall i$.  We note that the system established in (\ref{mixed_complementarity_GNEP}) will be solved repeatedly for different right-hand side $f_i$ vectors to generate our simulation data.  This will be discussed further in Section \ref{sec:numerical-setup}.   

\subsection{General Residual Model \citep{ratliff2014social}}\label{sec:theory-residual}

\noindent Having explained that we will use a mixed complementarity form of our generalized Nash game (\ref{new_GNE}) to produce our simulation data by operating under the assumption of equal multipliers for the joint constraint \cite{rosen1965existence,harker1991generalized,facchinei2007generalized,facchinei2010generalized,gabriel2012complementarity}, we can explain the inverse optimization model we use to extract parameters from this simulation data.  We utilize the residual model showcased in \citet{ratliff2014social}, which is quite similar to the one presented in \citet{keshavarz2011imputing}, with the difference being that \citet{ratliff2014social} sum over players in a game.  As a note to the reader, we will explain this model utilizing the notation from \citet{ratliff2014social}'s more recent paper \citet{konstantakopoulos2017robust}. 

The model starts with the assumption that each player solves an optimization problem such as the following:
\begin{equation*}
    \min \{ f_i(x_i,x_{-i}) | x_i \in \mathcal{B}_i \}
\end{equation*}

\noindent with $f_i(x;\theta_i) = \langle \phi_i(x_i,x_{-i}), \zeta_i \rangle + \bar{f}_i(x)$ and $\mathcal{B}_i = \{x_i | h_{i,j}(x_i) \leq 0, j= 1,...,l_i,\ g_{i,q}(x) = 0,\ q = 1,...,ll_i \}$, in which $i$ represents the player, $l_i$ represents the number of inequality constraints for player $i$, $ll_i$ represents the number of equality constraints for player $i$, $\zeta_i$ represents the linear parameterization of the objective function for player $i$, and $\bar{f}_i$ is the ``known'' term in the objective function \citep{konstantakopoulos2017robust}.  In our game, each player also has one joint constraint inequality in the set $\mathcal{B}_i$ through which the players can influence each others' feasible regions.  Extending this model to accommodate this joint constraint will be covered in Section \ref{sec:theory-gnep}. 

If we take the KKT conditions of this problem \citep{boyd2004convex}, we obtain the following relevant ``residuals'' \citep{ratliff2014social,konstantakopoulos2017robust,keshavarz2011imputing}, in which $D_i$ represents the derivative with respect to $x_i$:

\begin{subequations}
\begin{equation}
    r_{s,i}^{(k)} (\zeta_i,\mu_i,\nu_i) = D_i f_i\left(x_i^{(k)},x_{-i}^{(k)}\right) + \sum\limits_{j=1}^{l_i} \mu_i^j D_i h_{i,j}\left(x_i^{(k)}\right) + \sum\limits_{q=1}^{ll_i} \nu_i^q D_i g_{i,q} \left( x_i^{(k)} \right)
\end{equation}
\begin{equation}
    r_{c,i}^{j,(k)}(\mu_i) = \mu_i^j h_{i,j}\left( x_i^{(k)} \right),\ j \in \{ 1,...,l_i \}
\end{equation}
\end{subequations}

\noindent with $k$ representing the ``kth observation'' in the game, $s$ representing the label for the stationarity residual, and $c$ representing the label for the complementarity residual \citep{konstantakopoulos2017robust,ratliff2014social}.  We then have $\mathbf{x}^{(k)}$ data for each instance $k$ of the game as a composite vector containing all the players' variable values, and we attempt to choose $\zeta, \mu, \nu$ such that we minimize the residuals presented above.  This leads to the optimization problem proposed by \citet{ratliff2014social} and \citet{konstantakopoulos2017robust} inspired by \citet{keshavarz2011imputing}:

\begin{subequations} \label{residual_model}
\begin{equation}
    \min\limits_{\mu,\zeta,\nu} \sum\limits_{i=1}^N \sum\limits_{k=1}^{\eta_i} \chi_i \left(r_{s,i}^{(k)}(\zeta,\mu,\nu), r_{c,i}^{(k)} (\mu) \right)
\end{equation}
\begin{equation}
    \zeta_i \in Z_i,\ \mu_i \geq 0,\ \forall i \in \{1,...,N \}
\end{equation}
\end{subequations}

\noindent with $N$ representing the total number of players and $\eta_i$ representing the number of times a player $i$ engages in the game which, in our application, means the number of origin-destination runs the player undertakes.  We assume in our implementation that $\eta_i$ is equal across all players, meaning that each player engages in each iteration of the game.  It is also important to note that there are separate parameterizations for each player $i$ in the form of $\zeta_i$.  Furthermore, these parameterizations each belong to their own convex set $Z_i$.  Finally, we specify that $\chi_i$ is a ``non-negative convex penalty function'' \citep{konstantakopoulos2017robust,ratliff2014social} which, according to \cite{boyd2004convex}, is applied to each set of residuals such that a sum of residuals is formed for each set of residuals. 

\subsection{Generalized Nash Connection to Residual Model}\label{sec:theory-gnep}

With background on the conversion of the generalized Nash problem to a mixed complementarity form and the use of the residual model for inverse optimization, we presently discuss extending the residual inverse optimization model to accommodate the joint constraint in our generalized Nash equilibrium problem.  Before we present our corollary regarding incorporating joint constraints into the inverse optimization model, we state a theorem from  \citet{facchinei2007generalized} in modified notation from \citet{facchinei2010generalized}.

This theorem requires the following specification of $\mathbf{X}$ for the purposes of the theorem as:
\begin{equation}\label{X_defined_with_h}
    \mathbf{X} = \{\mathbf{x} \in \mathbb{R}^{\bar{n}}: h_q(\mathbf{x}) \leq 0 \ q=1,...,m,\ g_j(\mathbf{x}) = 0\ j=1,...,p\}
\end{equation}

\noindent where $m$ represents in this case the total set of all inequality constraints across all players and $p$ represents the total set of equality constraints across all players. In the set $\mathbf{X}$, we incorporate all of the constraints for all of the players, including one copy of the joint constraints.  We could also write the subset of constraints for each player $v$ as:
\begin{equation}
    X_v(\mathbf{x}^{-v}) = \{x^v : h_q(x^v,\mathbf{x}^{-v}) \leq 0\ q=1,...,m,\ g_j(x,\mathbf{x}^{-v}) = 0\ j=1,...,p  \}
\end{equation}


\noindent For each of these subsets $X_v(\mathbf{x}^{-v})$, all of the constraints for all of the players still exist, but we are differentiating between player $v$'s variables and the other players' variables more explicitly; each subset $X_v(\mathbf{x}^{-v})$ also has its own copy of the joint constraints.  The set $\mathbf{X}$ and the subsets $X_v(\mathbf{x}^{-v})\ \forall v$ lead to two statements of the KKT conditions.  First for the $X_v(\mathbf{x}^{-v})$ subsets, according to \cite{facchinei2007generalized,facchinei2007finite,facchinei2010generalized}, if we assume that a constraint qualification holds for each player, a solution for each player $v$ will correspond with a KKT point that satisfies the following KKT conditions for each individual player $v$:
\begin{subequations} \label{KKT_GNEP}
\begin{equation} \label{KKT_GNEP_stat}
    \nabla_{x^v} \theta_v(x^v,\textbf{x}^{-v}) + \sum\limits_{q=1}^m \lambda_q^v \nabla_{x^v} h_q(x^v,\textbf{x}^{-v}) + \sum\limits_{j=1}^p \nu_j^v \nabla_{x^v} g_j(x^v,\textbf{x}^{-v}) = 0
\end{equation}
\begin{equation} \label{KKT_GNEP_comp1}
    0 \leq \lambda_q^v \ \bot \ h_q(x^v, \textbf{x}^{-v}) \leq 0,\ \forall q
\end{equation}
\begin{equation}\label{KKT_GNEP_comp2}
    g_j(x^v, \textbf{x}^{-v}) = 0,\ \nu_j^v \text{ free},\ \forall j
\end{equation}
\end{subequations}

\noindent Second for the $\mathbf{X}$ set, we know from \cite{facchinei2007finite,bazaraa2013nonlinear,facchinei2007generalized} that, if a constraint qualification is satisfied, then the KKT conditions are satisfied by a solution to the VI (assuming the presence of relevant multipliers) with $\mathbf{F}$ defined as (\ref{specific_QVI_formulation}) and $\mathbf{X}$ defined as (\ref{X_defined_with_h}).  We also know from \cite{facchinei2007finite} that, if $h_q$ are convex and $g_j$ are affine, then a solution to the KKT conditions is a solution to the VI with $\mathbf{F}$ defined as (\ref{specific_QVI_formulation}) and $\mathbf{X}$ defined as (\ref{X_defined_with_h}).  Therefore, we have the resulting KKT conditions from the VI as: 
\begin{subequations} \label{KKT_VI}
\begin{equation} \label{KKT_VI_stat}
    \mathbf{F}(\textbf{x}) + \sum\limits_{q=1}^m \lambda_q \nabla_{\textbf{x}} h_q(\textbf{x}) + \sum\limits_{j=1}^p \nu_j \nabla_{\mathbf{x}} g_j(\mathbf{x}) = 0
\end{equation}
\begin{equation}\label{KKT_VI_comp1}
    0 \leq \lambda_q \ \bot \ h_q(\textbf{x}) \leq 0, \forall q
\end{equation}
\begin{equation}\label{KKT_VI_comp2}
    g_j(\mathbf{x}) = 0,\ \nu_j \text{ free},\ \forall j
\end{equation}
\end{subequations}

\noindent We can now convey Theorem 4.8 from \cite{facchinei2010generalized} who cite \cite{facchinei2007generalized,harker1991generalized} which states that we can move between the two forms of the KKT conditions (\ref{KKT_GNEP}) and (\ref{KKT_VI}) so long as we assume the multipliers $\lambda^v, \nu^{v}$ for (\ref{KKT_GNEP}) are equal across all players $v$. This statement of the theorem is slightly different from either \cite{facchinei2007generalized}'s statement or \cite{facchinei2010generalized}'s statement because we have added equality constraints and their associated multipliers into our formulation of these KKT conditions.  

\begin{theorem}[Theorem 4.8 \citep{facchinei2010generalized}]\label{thm:gnep-facchinei} 
Consider the jointly convex GNEP with $h_q\ \forall q,g_j\ \forall j,\theta_v$ being $C^1$.  Then the following statements hold:
\begin{itemize}
    \item Let $\bar{\mathbf{x}}$ be a solution of the VI$(\mathbf{X},\mathbf{F})$ (with $\mathbf{X}$ defined as (\ref{X_defined_with_h}), $\mathbf{F}$ defined as (\ref{specific_QVI_formulation}), and VI defined as (\ref{VI_formulation})) such that the KKT conditions (\ref{KKT_VI}) hold with some multipliers $\bar{\lambda}$ and $\bar{\nu}$.  Then $\bar{\mathbf{x}}$ is a solution of the GNEP, and the corresponding KKT conditions (\ref{KKT_GNEP}) are satisfied with $\lambda^1 = ... = \lambda^N = \bar{\lambda}$ and with $\nu^1 = \nu^2 = ... = \nu^N = \bar{\nu}$.
    
    \item Conversely, assume that $\bar{\mathbf{x}}$ is a solution of the GNEP such that the KKT conditions (\ref{KKT_GNEP}) are satisfied with $\lambda^1 = ... = \lambda^N$ and with $\nu^1 = \nu^2 = ... = \nu^N$.  Then $(\bar{\mathbf{x}},\bar{\lambda}, \bar{\nu})$ with $\bar{\lambda} = \lambda^1$ and with $\bar{\nu} = \nu_1$ is a KKT point of VI$(\mathbf{X},\mathbf{F})$, and $\bar{\mathbf{x}}$ itself is a solution of VI$(\mathbf{X},\mathbf{F})$.
\end{itemize}

\end{theorem}

\begin{proof}
As a brief proof of this theorem (which is based off of, and is a relatively direct extension of, a related proof by~\citet{facchinei2007generalized}), for the first statement, we notice a correspondence between (\ref{KKT_GNEP_stat}) concatenated for all players $v$ and (\ref{KKT_VI_stat}) as long as the multipliers are equal across players and are equal to the $\bar{\lambda}, \bar{\nu}$ multipliers.  This correspondence continues for (\ref{KKT_GNEP_comp1}) \& (\ref{KKT_VI_comp1}) and (\ref{KKT_GNEP_comp2}) \& (\ref{KKT_VI_comp2}) so long as the multipliers are equal.  Therefore, the $(\bar{\mathbf{x}}, \bar{\lambda}, \bar{\nu})$ from (\ref{KKT_VI}) is a KKT point for the conditions (\ref{KKT_GNEP}) for each player.  According to \cite{facchinei2007generalized}, as long as the conditions (\ref{KKT_GNEP}) are sufficient for optimality for each player, which is true if we assume $\mathbf{F}$ is convex, $h_q$ are convex, and $g_j$ are affine \cite{boyd2004convex}, $\bar{\mathbf{x}}$ with multipliers $\bar{\lambda}, \bar{\nu}$ is a solution for each player's optimization problem, thus making it a  GNEP solution.

For the second statement, we point to the fact that, now that we are given that the multipliers are equal across all of the players, then there is a correspondence between (\ref{KKT_GNEP}) (concatenated for all players $v$ for the stationarity constraint) and (\ref{KKT_VI}) as long as $\bar{\lambda} = \lambda_1 = ... = \lambda_N$ and $\bar{\nu} = \nu_1 = ... = \nu_N$.  Consequently, we have a KKT point for the (\ref{KKT_VI}) conditions that comes from the given solution to (\ref{KKT_GNEP}).  We know from \cite{facchinei2007finite} that if there is a KKT point for the conditions (\ref{KKT_VI}) and we assume the constraints $h_q$ are convex and $g_j$ are affine, then the $\bar{\mathbf{x}}$ associated with the KKT point is a solution for the VI. 
\end{proof}

\noindent Our technique, as discussed in Section~\ref{sec:prelims}, can be looked at as a corollary to this theorem, where we state that we can use the stacked KKT conditions with one copy of the joint/shared constraints to parameterize the cost functions of the players.  

\begin{corollary}\label{thm:gnep-allen}
With the assumption of equal multipliers for the joint constraint(s) in a GNEP, the stationarity conditions and complementarity conditions of the VI KKT conditions (\ref{KKT_VI}) may be used to form a residual model of the form seen in \cite{keshavarz2011imputing,ratliff2014social,konstantakopoulos2017robust}, with $\chi$ as a ``non-negative convex penality function'' as in those papers and in \cite{boyd2004convex} and with $k$ representing the multiple data points $\mathbf{x}^k$: \begin{subequations}\label{corollary:model}
    \begin{equation}\label{corollary:objective_function}
    \min\limits_{\zeta_i,\lambda_q,\nu_j} \sum\limits_k \chi\left(\mathbf{F}(\textbf{x}^k) + \sum\limits_{q=1}^m \lambda_q \nabla_{\textbf{x}} h_q(\textbf{x}^k) + \sum\limits_{j=1}^p \nu_j \nabla_{\mathbf{x}} g_j(\mathbf{x}^k)\right) + \sum\limits_{k,q} \chi\left(\lambda_q h_q(\textbf{x}^k)\right)
    \end{equation}
    \begin{equation}
        \lambda_q \geq 0\ \forall q,\ \zeta_i \in Z_i\ \forall i
    \end{equation}
    \end{subequations}
\end{corollary}

\begin{proof}
We know from Theorem \ref{thm:gnep-facchinei} that there is a correspondence between the solutions to (\ref{KKT_GNEP}) for all players $\nu$ and the solution to (\ref{KKT_VI}) so long as the multipliers are equal across these two sets of KKT conditions.  Therefore, we can take the stationarity condition and the complementarity conditions of (\ref{KKT_VI}) as the residuals for the GNEP problem and form the same kind of residual problem as (\ref{residual_model}). 
\end{proof}



\noindent As a result of corollary \ref{thm:gnep-allen}, we can find parameterizations for players' cost functions in generalized Nash games with joint constraints.  Specifically, in our transportation problem, we can find $\mathcal{C}_i$ and $\bar{c}_i$ for players $i \in [N]$.  Next, in Section 4, we explain our experimental framework, walking the reader through our simulation set-up and the specifications for our residual model for the transportation problem.

\section{Experimental Framework}\label{sec:numerical-setup}

To test and explore our inverse optimization framework for generalized Nash Equilibrium games, we utilize various grid networks, which can be expressed as a certain number of nodes vertically and a certain number of nodes horizontally (2 nodes by 2 nodes, 3 nodes by 3 nodes, 4 nodes by 4 nodes, and 5 nodes by 5 nodes), and the Sioux Falls network \cite{sioux_falls_data,leblanc1975efficient} for the traffic problem presented in the introduction; examples of these networks are visualized in Figures \ref{fig:network-grid} and \ref{fig:network-sioux-falls}, respectively.

\begin{figure}[h!]
\begin{subfigure}{0.4\textwidth}
\centering
\includegraphics[height=0.2\textheight,keepaspectratio]{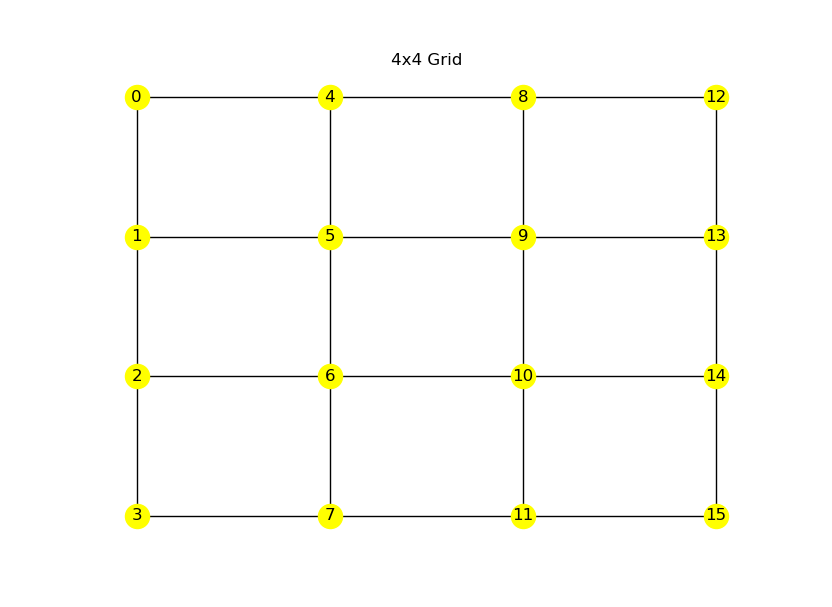}
\caption{4 nodes x 4 nodes Grid}\label{fig:network-grid}
\end{subfigure}
\begin{subfigure}{0.5\textwidth}
\centering
\includegraphics[height=0.3\textheight,keepaspectratio]{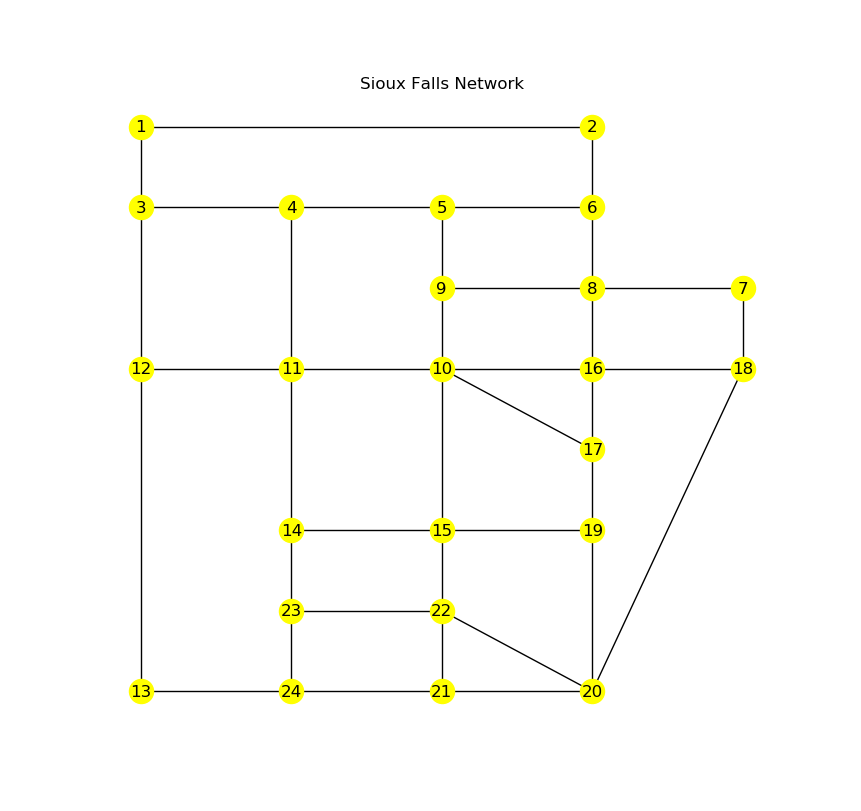}
\caption{Sioux Falls Network \cite{sioux_falls_data,leblanc1975efficient}}\label{fig:network-sioux-falls}
\end{subfigure}
\caption{Networks utilized for testing inverse optimization framework}
\end{figure}

To generate the solutions to use in our inverse optimization framework, we utilize a slightly varied form of the KKT conditions listed above in (\ref{mixed_complementarity_GNEP})  
in which we substitute the stationarity constraint (\ref{MCP_stationarity}) into the nonnegative complementarity constraints (\ref{MCP_nonneg}) thus producing (\ref{new_MCP_nonneg}) \citep{boyd2004convex,gabriel2012complementarity,facchinei2010generalized}:

\begin{subequations}\label{new_MCP}
\begin{equation}\label{new_MCP_nonneg}
    0 \leq x_i \ \bot \ \mathcal{C}_i \left(2x_i + \sum\limits_{j\neq i} x_j\right) + \bar{c}_i + D^T v_i + \bar{u} \geq 0 \ \forall i = 1,...,N
\end{equation}
\begin{equation}
    0 \leq \alpha-\left(\sum\limits_j x_j\right) \ \bot \ \bar{u} \geq 0
\end{equation}
\begin{equation}
    D x_i = f_i,\ v_i \text{ free}\ \forall i = 1,...,N
\end{equation}
\end{subequations}

\noindent We utilize the implementation of PATH \citep{dirkse1995path,ferris2000complementarity} from the software GAMS \cite{GAMS_software_33_2,GAMS_software_34_1} for this problem.

As stated in the introduction, our goal is to solve the inverse optimization residual model for points generated from the GNE problem, specifically with the purpose of finding the parameters $\mathcal{C}_i$ and $\bar{c}_i$ for each player $i$.  In some of the experiments, these $\mathcal{C}_i$ and $\bar{c}_i$ will be assumed the same across all players and, in others, they will be assumed different across all players.  As a reminder of the residual model original presented in Section \ref{sec:prelims}, we present it again here \cite{keshavarz2011imputing,ratliff2014social,konstantakopoulos2017robust}:  

\begin{subequations}\label{empirical_residual_model_2}
\begin{small}
\begin{equation*}
        \min\limits_{\mathcal{C}_i,\bar{c}_i,v_i^k,u_i^k,\bar{u}^k} \sum\limits_{k} \left( \overbrace{\sum\limits_{i=1}^N ||\mathcal{C}_i \left(2 x_i^k + \sum\limits_{j\neq i}x_{j}^k\right) + \bar{c}_i + D^T v_i^k - u_i^k + \bar{u}^k||_1}^{\text{stationarity}} + \overbrace{\sum\limits_{i=1}^{N} ||(x_i^k)^T u_i^k ||_1}^{\text{complementarity \#1}} + \overbrace{||(\alpha-\sum\limits_{j}x_{j}^k)^T(\bar{u}^k)||_1}^{\text{complementarity \#2}} \right)
\end{equation*}
\end{small}
\begin{equation}
    L_1 \leq diag(\mathcal{C}_i) \leq U_1 \ \forall i
\end{equation}
\begin{equation}
    L_2 \leq \bar{c}_i \leq U_2 \ \forall i
\end{equation}
\begin{equation}
    \bar{u}^k \geq 0\ \forall k\ u_i^k \geq 0\ \forall i,k
\end{equation}
\end{subequations}


\noindent  In this mathematical program, we minimize the error across the relevant KKT conditions for all players and all data such that the diagonal of the $\mathcal{C}_i$ matrices and the $\bar{c}_i$ vectors remain between the two bounds.  The $L_1,L_2 \in \mathbb{R}^n$ and $U_1,U_2 \in \mathbb{R}^n$ bounds are determined by the experimentation bounds for those variables; for example, if we randomly choose the diagonal of $\mathcal{C}_i$ to uniformly come from between 1 and 5, then the upper and lower bounds for the diagonal will be 1 and 5.  We use the 1-norm to preserve the convexity of the problem and to create linear terms as discussed in Section \ref{sec:prelims}.

For our experiments, we follow a general framework in which various elements can be modified.  The general framework is as follows:


\begin{algorithm}[H]
\textbf{Input:} $D$ conservation-of-flow matrix, $n$ nodes, $m$ arcs, $\alpha$ capacity level, $N$ number of players, bounds $L_1,L_2,U_1,U_2$ \\
Exhaustively generate $\mathcal{P}$, the set of all unique origin-destination pairs for a given network \\
\For{each pair in $\mathcal{P}$}{
Solve the mixed complementarity problem (\ref{new_MCP}) form of GNEP (\ref{new_GNE}) for RHS $f_i$ $\forall i$ constructed from the current pair, with one unit of flow for each player  \\
Store the $N$ player flow vectors $x_i$ $i=1,...,N$ (which are the data points) \\
}
Input the sets of data points into the IO mathematical program (\ref{empirical_residual_model_2}) to obtain estimates for $\mathcal{C}_i$ and $\bar{c}_i\ \forall i$
\caption{Basic Experimental Framework}
\end{algorithm}

\vspace{1em}


\noindent  We again note that, in some experiments, we assume $\mathcal{C}_i$ and $\bar{c}_i$ are the same across all $i$ and, in others, we assume $\mathcal{C}_i$ and $\bar{c}_i$ are different across all $i$.  We also utilize all $\mathcal{P}$ origin-destination pairs for a given network because we want to gain a detailed picture of flow on that network.  However, we recognize that there are many more configurations of possible flow from which we could have drawn which may lead to further study, including configurations that would involve the players starting and ending at different points.





\section{Experimental Results}\label{sec:experiments}

In this section, we present the numerical results of experiments with the proposed inverse optimization framework to parameterize the GNEP (\ref{new_GNE}) using the two types of networks showcased in Section~\ref{sec:numerical-setup}.  To implement the experimental setup described in Section \ref{sec:numerical-setup}, we used a data generation and optimization pipeline including state of the art solvers from PATH \cite{dirkse1995path,ferris2000complementarity} in GAMS and Gurobi \cite{gurobi_citation}.  This included randomly generating from a uniform distribution the original costs for the $\mathcal{C}_i$ and $\bar{c}_i$ parameters in the simulation model outlined in (\ref{new_MCP}).  See Appendix \ref{appendix:experimental_details} for more information about the experimental setup. 

The boxplots\footnote{For all of the box plots, the ``whiskers'' are placed at quartile 1 - 1.5 (quartile 3 - quartile 1) and at quartile 3 + 1.5 (quartile 3 - quartile 1), and the ``dots'' are outliers~\citep{matplotlib_boxplot}.  
} in this section denote two different measures of error. First, there is the value of the objective function of the residual model, denoting how closely the chosen parameterizations cause the input data to satisfy the KKT conditions \cite{keshavarz2011imputing,ratliff2014social,konstantakopoulos2017robust}. 
Second, there is the total flow error metric, which measures the difference between the flows under both the original parameterization and the IO residual model parameterization using the same set of OD pairs $\mathcal{P}$.  The Frobenius-norm metric is utilized, which acts as a vector 2-norm for matrices.  This is useful compared to other matrix norms because it calculates a total difference between the two matrices, not a maximal difference along rows or columns or an abstract eigenvalue metric as in the case of the more traditional matrix norms \cite{matrix_norm}.  This leads to the following metric.
\begin{definition}[Flow error]\label{def:flow-error}
The \textbf{flow error} is calculated by first taking the squared difference between the flow under the original parameterization and the flow under the IO parameterization for all origin-destination pairings, all players, and all arcs.  Then, a sum is taken across these squared differences and, finally, the square root of this summation is taken.  This error represents the Frobenius norm between the two sets of flows.
\end{definition}
This error is also normalized by the total number of arc flows, which is calculated by multiplying the number of origin-destination pairings, the number of players, and the number of arcs.  It forms the following normalized metric.
\begin{definition}[Normalized flow error]\label{def:normalized-flow-error}
The \textbf{normalized flow error} is calculated by dividing the flow error by a factor created by multiplying the number of origin-destination pairings, the number of players, and the number of arcs.  It represents a per unit level of error. 
\end{definition}

Note that data are collected regarding the differences between the original parameterizations and the IO parameterizations, and it is found that there are differences between the two.  However, the important metrics are the objective function value and the flow error metrics because the first measures how well the IO parameterization fits the model and the model data and because the second showcases whether or not the IO parameterization leads to the same flow patterns observed for the OD pairs under the original parameterization. Note also that the objective function values for the experiments can be found in Appendix \ref{appendix:obj_func_for_experiments} and timing information for the experiments can be found in Appendix \ref{appendix:timing_for_experiments}.

\subsection{Broader Observations about the Results}

Overall, the results from the four experimental groups (with two in the same cost randomized costs category and two in the different cost randomized costs category) are encouraging. The maximum objective function values for all the experiments across the groups are on the order of 1e-6, and the maximum flow errors (see Definition~\ref{def:flow-error}) for all of the experiments across the groups are on the order of 1e-6 or 1e-7.  Therefore, these error metrics indicate that viable parameterizations are being recovered for the OD pair sets under which we are testing the framework.  Indeed, the fact that these errors do not greatly differ between same and different costs is encouraging because it indicates that the approach can work to recover people's different perceptions of road networks \cite{bertsimas2015data,thai2015multi,thai2018imputing,xu2018network,chow2012inverse,chow2015activity}. However, for the different costs case, it should be noted that all of the $\mathbf{F}$ functions as defined by (\ref{specific_QVI_formulation}) across the number of arcs present in the grids (2x2-5x5) and in Sioux Falls as well as the number of players, with the exception of one trial in the 5x5 grid and 10 players case, are strongly monotone functions.  As indicated in Lemma \ref{lemma:not-all-C-equal}, we cannot make this guarantee in general for players with different interaction costs.  This matters because the different starting point scheme for solving the simulation model is bolstered by the positive definiteness assumption and, as can be seen in Appendix \ref{appendix:different_cost_eigenvalues}, there is less of a chance that the matrix will be positive definite as the number of players increases.  Therefore, this suggests further computational work that could be done to see if non-positive definite matrices of our type work in general in our simulation framework.\footnote{Note that the one non-positive definite matrix in the 5x5 grid with 10 players and $\alpha=10$ example appears to work in the simulation framework because the flow error and objective function values are both low.}  It should also be noted that, for different costs, the experiments with $\alpha=5$ and $N=10$ for the 5x5 grid and Sioux Falls did not solve completely, with the 5x5 grid experiment only able to complete 6 iterations and the Sioux Falls experiment unable to complete any at the time of this posting. While Lemma \ref{complexity_lemma} shows that problem (\ref{empirical_residual_model_2}) is solvable in polynomial time, any solver will naturally encounter scalability problems as the problem size increases.  However, the reality that the framework did work for most of the experiments under restrictive $\alpha$ values of $(0.5) N$ validates the extension of the framework, because some of the OD pairs involved start, end, or both nodes where only two arcs were coming out of or into the node, which meant the capacity constraints were guaranteed to be tight. 


\subsection{Same Randomized Costs}

Table \ref{same_costs_experiment_details} details each of the experiment groups for this subsection and provides an experiment group number to which we will refer in this section.  In this subsection, we assume that $\mathcal{C}_i$ and $\bar{c}_i$ are equal across all players $i$ such that $\mathcal{C}_1 = \mathcal{C}_2 = ... = \mathcal{C}_N = \mathcal{C}$ and $\bar{c}_1 = \bar{c}_2 = ... = \bar{c}_N = \bar{c}$.  

\begin{table}[H]
    \centering
    \begin{tabular}{c|c|c|c|c|c|c|c} \textbf{\#} & \textbf{Network} & \textbf{Costs} & \textbf{Players ($N$)} & \textbf{\# Trials} & $\mathbf{\mathcal{C}}$ \textbf{Bounds} & $\mathbf{\bar{c}}$ \textbf{Bounds} & $\mathbf{\alpha}$ \\ \hline
    1 & 2x2-5x5 & Same & 2, 5, 10 & 10 & [1,5] & [5,20] & $0.5(N)$, $N$ \\
    2 & Sioux Falls & Same & 2, 5, 10 & 10 & [1,5] & [5,20] & $0.5(N)$, $N$
    \end{tabular}
    \caption{Experimental details for the ``Same Costs'' setting.  Here, \# refers to the experiment group number, Network refers to the graph upon which the experiment was run, $\mathcal{C}$ Bounds indicating the range for randomization and the bounds used in the IO mathematical program for the diagonal of the $\mathcal{C}$, $\bar{c}$ Bounds indicates the same for the $\bar{c}$ parameters, and $\alpha$ refers to that parameter value}
    \label{same_costs_experiment_details}
\end{table}

Experiment group 1 (Grid) iterates over generated 2x2, 3x3, 4x4, and 5x5 grids, each with 10 trials of randomly chosen parameters, which the inverse optimization residual model attempts to estimate.  For each grid, three different player numbers ($N=2,5,10$) and two different $\alpha$ values, one set to half the number of players for each $N$ and one set to exactly the number of players for each $N$, are considered.  The graphs for this experiment group display many box plots, and the system utilized for labeling the box plots is Grid Size/Number of Players/Alpha Value. Each set of eight moving from left to right indicate the same number of players.  With regard to flow error, Figure \ref{fig:flow_error_experiment_1} demonstrates that, for all of the subsets of players and $\alpha$ values, as grid size increases, the total flow error also increases.  It is important to note that everything on this graph is still on the order of 1e-7, but this trend is also evident.  It likely results from the fact that, as grid size increases, the number of arcs and number of OD pairs also increases and, since this measure is calculated across all OD pairs, all players, and all arcs, then even if consistent error were assumed across all arcs, the norm would have to increase. Indeed, upon examining the accompanying Figure \ref{fig:normalized_flow_error_experiment_1}, it is apparent that the normalized flow error, calculated as described in Definition \ref{def:normalized-flow-error}, decreases as grid size increase across all of the subsets of players and $\alpha$ values. 


\begin{figure}[h]
\begin{subfigure}[b]{0.8\textwidth}
\centering
\includegraphics[width=\textwidth,keepaspectratio]{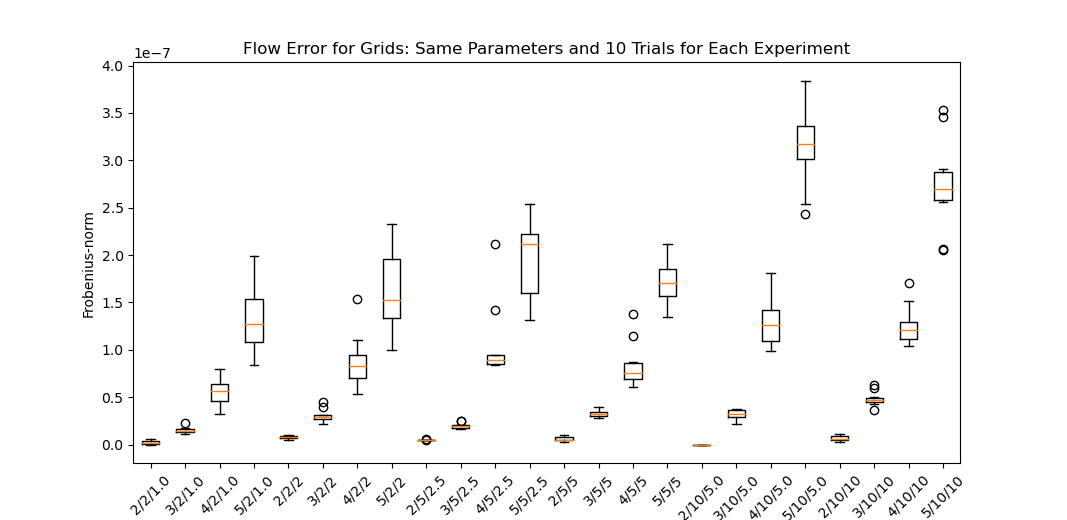}
\caption{Flow Error for Experiment Group 1}\label{fig:flow_error_experiment_1}
\end{subfigure}
\begin{subfigure}[b]{0.8\textwidth}
\centering
\includegraphics[width=\textwidth,keepaspectratio]{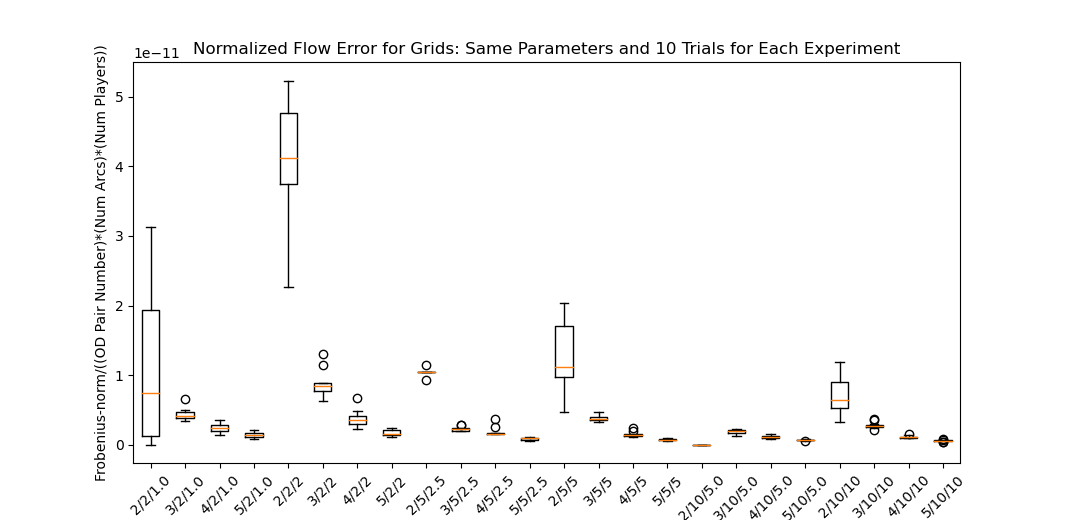}
\caption{Normalized Flow Error for Experiment Group 1}\label{fig:normalized_flow_error_experiment_1}
\end{subfigure}
\caption{Flow Error Metrics for Experiment Group 1: The labeling at the bottom of the graphs indicates attributes of the boxplot, specifically the Grid Size/Number of Players/Alpha Value}\label{fig:experiment_1_flow_error_metrics}
\end{figure}

Experiment group 2 (Sioux Falls) again consists of 10 trials for each of $N=2,5,10$ players and the two $\alpha=(0.5)N,N$ values.  The labeling for the graphics is set to Number of Players/Alpha Value.  In examining flow errors, Figure \ref{fig:flow_error_experiment_2} showcases that the flow errors do not appear to follow a set pattern when moving between the different numbers of players, yet they are all still small and on the order of 1e-7 or lower. However, in Figure \ref{fig:normalized_flow_error_experiment_2}, the median flow error decreases as number of players increases, even on this very small scale (1e-12).   

\begin{figure}[h!]
\begin{subfigure}[b]{0.49\textwidth}
\centering
\includegraphics[width=\textwidth,keepaspectratio]{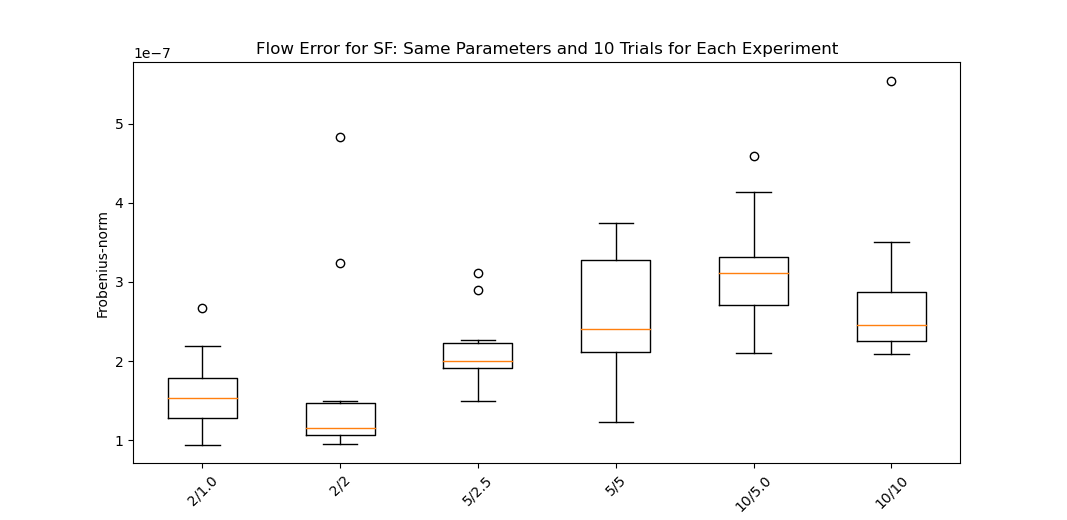}
\caption{Flow Error for Experiment Group 2}\label{fig:flow_error_experiment_2}
\end{subfigure}
\begin{subfigure}[b]{0.49\textwidth}
\centering
\includegraphics[width=\textwidth,keepaspectratio]{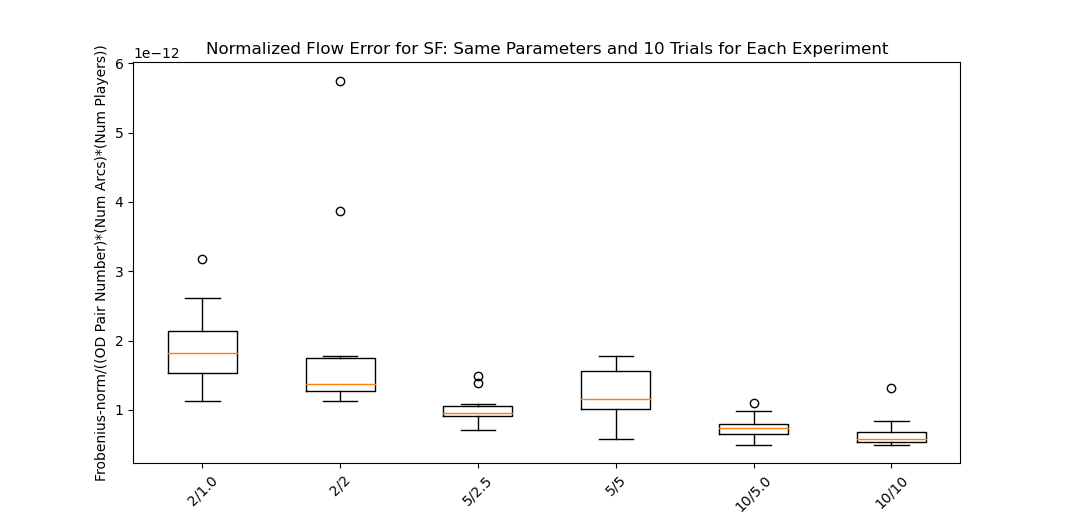}
\caption{Normalized Flow Error for Experiment Group 2}\label{fig:normalized_flow_error_experiment_2}
\end{subfigure}
\caption{Flow Error Metrics for Experiment Group 2: The labeling at the bottom of the graphs indicates attributes of the boxplot, specifically the Number of Players/Alpha Value}\label{fig:experiment_2_flow_error_metrics}
\end{figure}


\subsection{Different Randomized Costs}
Next, we assume that the costs are not the same across all of the players, meaning that a new $C_i$ and a new $\bar{c}_i$ are drawn for each player $i$.  Table~\ref{different_costs_experiment_details} describes each experiment group in this subsection; as in the previous subsection, an experiment group number is assigned to the two sets of experiment groups.

\begin{table}[h]
    \centering
    \begin{tabular}{c|c|c|c|c|c|c|c} \textbf{\#} & \textbf{Network} & \textbf{Costs} & \textbf{Players ($N$)} & \textbf{\# Trials} & $\mathbf{\mathcal{C}_{i}}$ \textbf{Bounds} & $\mathbf{\bar{c}_i}$ \textbf{Bounds} & $\mathbf{\alpha}$ \\ \hline
    3 & 2x2-5x5 & Different & 2, 5, 10 & 10 & [1,5] & [5,20] & $0.5(N)$, $N$ \\
    4 & Sioux Falls & Different & 2, 5, 10 & 10 & [1,5] & [5,20] & $0.5(N)$, $N$
    \end{tabular}
    \caption{Tables with Experiment Details for Different Costs.  Note that \# refers to the experiment group number, Network referring to the graph upon which the experiment was run, $\mathcal{C}_{i}$ Bounds indicating the range for randomization and the bounds used in the IO mathematical program for the diagonal of the $\mathcal{C}_i$, $\bar{c}_i$ Bounds indicates the same for the $\bar{c}_i$ parameters, and $\alpha$ refers to that parameter value}
    \label{different_costs_experiment_details}
\end{table}

Experiment group 3 is a mirror image of experiment group 1, except that all of the players do not have the same costs.  One issue with this experiment group was that not all of the trials finished for the case of the 5x5 grid, 10 players, and $\alpha=5$; only 6 of the trials finished in under 24 hours,\footnote{This means about 24 hours were given for each trial before the trial was stopped, with the exception of one that was stopped by the computer before convergence.} so they are the ones included in the box plots.  Similar to the flow error for the grids under the same costs (Figure \ref{fig:flow_error_experiment_1}), Figure \ref{fig:flow_error_experiment_3} demonstrates that the median flow error increases as the grid size increases for all the subsets of player number and $\alpha$ value.  Figure \ref{fig:normalized_flow_error_experiment_3} illustrates much of the same decrease in median normalized flow as grid size increases (among the subsets) as in Figure \ref{fig:normalized_flow_error_experiment_1} for same costs across all players.  

\begin{figure}[h!]
\begin{subfigure}[b]{0.49\textwidth}
\centering
\includegraphics[width=\textwidth,keepaspectratio]{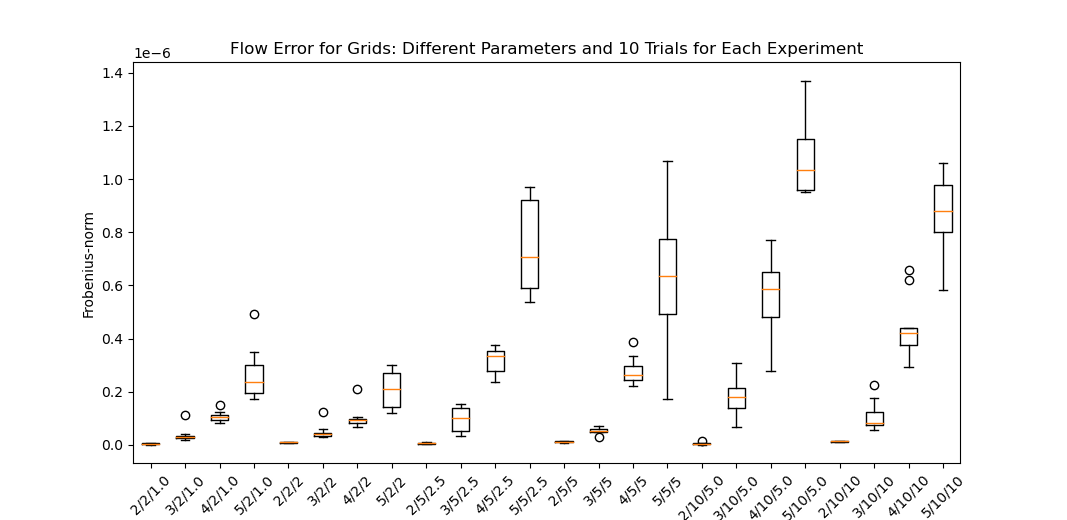}
\caption{Flow Error for Experiment Group 3}\label{fig:flow_error_experiment_3}
\end{subfigure}
\begin{subfigure}[b]{0.49\textwidth}
\centering
\includegraphics[width=\textwidth,keepaspectratio]{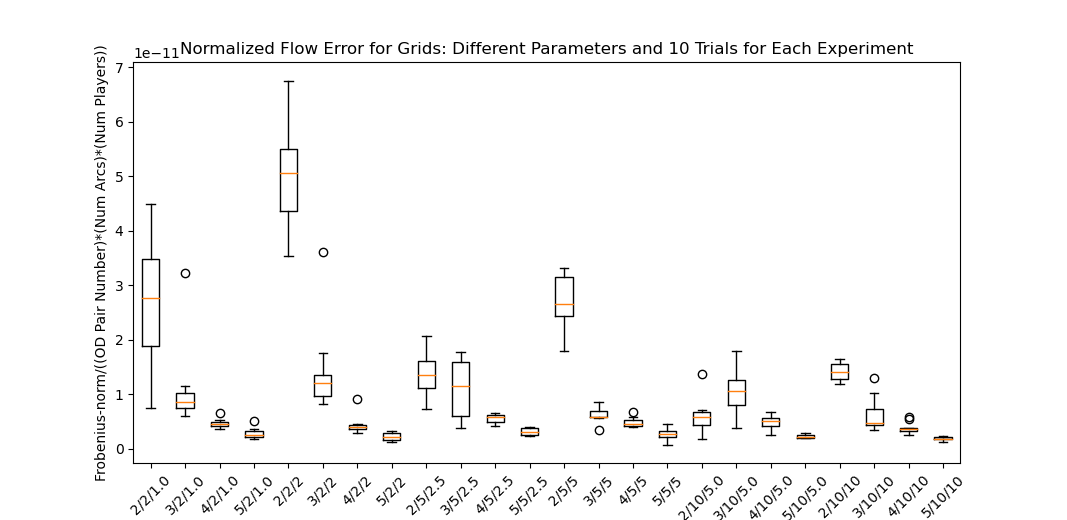}
\caption{Normalized Flow Error for Experiment Group 3}\label{fig:normalized_flow_error_experiment_3}
\end{subfigure}
\caption{Flow Error Metrics for Experiment Group 3: The labeling at the bottom of the graphs indicates attributes of the boxplot, specifically the Number of Players/Alpha Value.  Note: Only 6 trials are included for 5/10/5.0, see note in the text.}\label{fig:experiment_3_flow_error_metrics}
\end{figure}

Experiment group 4 (Sioux Falls) again consists of 10 trials for each of $N=2,5,10$ players and the two $\alpha=(0.5)N, N$ values.  The labeling for the graphics is Number of Players/Alpha Value.  It should be noted that the trials did not finish at the time of this posting for the $N=10$ and $\alpha=5$ experiment, so those boxplots are not included in Figure \ref{fig:experiment_4_flow_error_metrics}.  As concerns flow error, Figure \ref{fig:flow_error_experiment_4} shows increasing median flow error as player number increases, which was not as visible in Figure \ref{fig:flow_error_experiment_2} for experiment group 2.  However, unlike Experiment group 2's (Sioux Falls, but in the same cost setting) Figure \ref{fig:normalized_flow_error_experiment_2}, Figure \ref{fig:normalized_flow_error_experiment_4} does not demonstrate the same consistent increase in median normalized flow error, with normalized median flow decreasing with $N=10$ and $\alpha=10$.

\begin{figure}[h]
\begin{subfigure}[b]{0.49\textwidth}
\centering
\includegraphics[width=\textwidth,keepaspectratio]{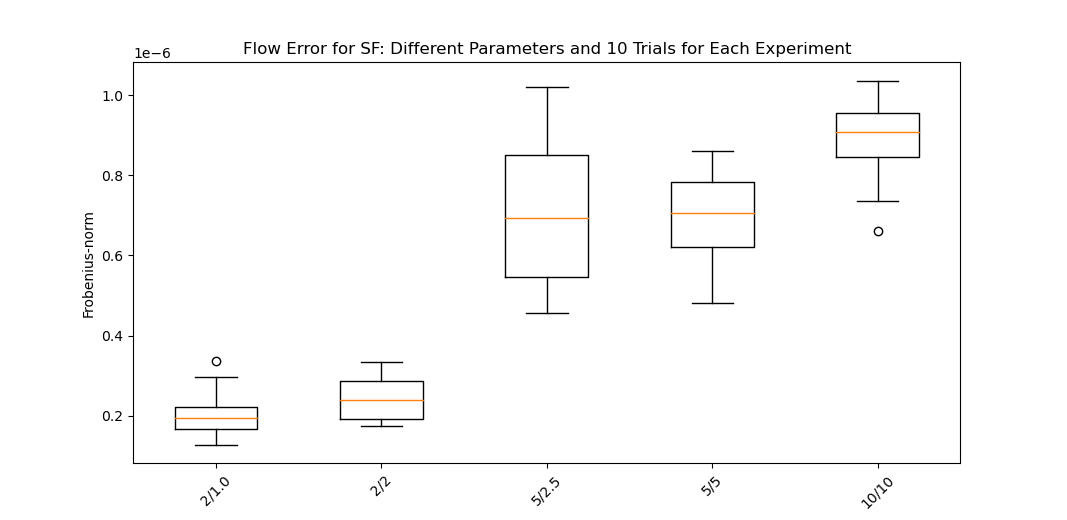}
\caption{Flow Error for Experiment Group 4}\label{fig:flow_error_experiment_4}
\end{subfigure}
\begin{subfigure}[b]{0.49\textwidth}
\centering
\includegraphics[width=\textwidth,keepaspectratio]{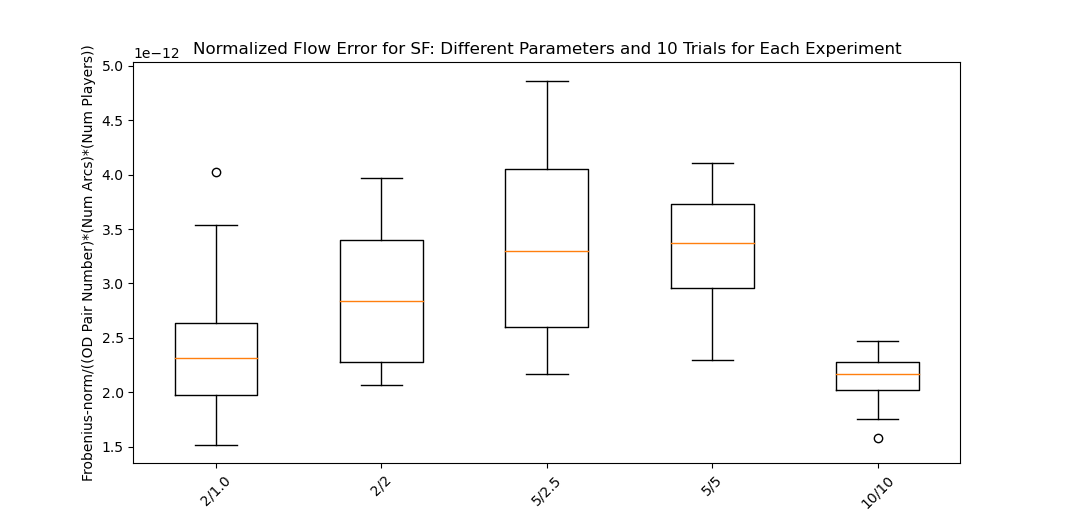}
\caption{Normalized Flow Error for Experiment Group 4}\label{fig:normalized_flow_error_experiment_4}
\end{subfigure}
\caption{Flow Error Metrics for Experiment Group 4: The labeling at the bottom of the graphs indicates attributes of the boxplot, specifically the Number of Players/Alpha Value.  Note: The trials did not finish in time for $N=10$ and $\alpha=5$, hence that boxplot is not included.}\label{fig:experiment_4_flow_error_metrics}
\end{figure}

\section{Conclusions \& Future Research}\label{sec:conclusions}

\citet{ratliff2014social} applied the framework due to \citet{keshavarz2011imputing} to multi-player Nash games.  We extended this framework to the case of generalized Nash games with joint/shared constraints by proving that we can use the VI KKT conditions 
in conjunction with a residual model
to recover a parameterization for the objective functions of the players in 
the multi-player game.
We have seen that, although our model may not recover the original parameterization, it recovers a parameterization that produces the same flow patterns as the original parameterization.  This holds true across multiple grid sizes, the Sioux Falls network, different assumptions regarding players' perceived costs, and the majority of restrictive $\alpha$ capacity settings and the associated numbers of players.

Further research could extend our model to the setting where the multipliers for the shared constraint are not assumed to be equal, especially in the case when the cost functions differ between the players \citep{gabriel2012complementarity}.  For more on this, we suggest to readers the work of \citet{nabetani2011parametrized}.  In addition, we could extend the research to a real-world traffic situation with real data, in which we would likely need to incorporate more players into our framework and different origin-destination traffic patterns, including ones in which not all the players leave from the same origin node and head to the same destination node.   Furthermore, we could attempt to address some of the issues with the Nash approach, including looking at a low regret condition on the players as in \cite{nekipelov2015econometrics,waugh2011computational,waugh2013computational}.  We could also explore making our model more robust to the issues indicated by \cite{peysakhovich2019robust} such as rationality assumptions and incorrect models. 

\begin{acks}
%
%
%
%
Allen was partially funded by a Graduate Fellowship in STEM Diversity while completing this research (formerly known as a National Physical Science Consortium Fellowship).  Allen was also supported by a Flagship Fellowship and a Dean's Fellowship from the University of Maryland-College Park and has worked for Johns Hopkins University Applied Physics Lab in the summers.

Dickerson was supported in part by NSF CAREER Award IIS-1846237, NSF D-ISN Award \#2039862, NSF Award CCF-1852352, NIH R01 Award NLM-013039-01, NIST MSE Award \#20126334, DARPA GARD \#HR00112020007, DoD WHS Award \#HQ003420F0035, and a Google Faculty Research Award.
%

Gabriel was supported by the following grant: Research Council of Norway, “Trans-Atlantic Cooperation on Energy Market Modeling,” 2018-2022.

We would like to thank Michael Curry of University of Maryland, College Park for helpful comments earlier on in this project.  We would like to thank Luke Evans of University of Maryland, College Park for reading an early draft of this paper. 
\end{acks}

\bibliographystyle{ACM-Reference-Format}
\bibliography{sample-file}


\appendix

\section{Supplementary Materials}

\subsection{Variable Number Calculations for Lemma \ref{complexity_lemma}}\label{appendix:variable_number}

In this subsection, we will provide the exact variable counts for problem (\ref{empirical_residual_model_1}) as the grid size increases from 2x2 to 3x3 to 4x4, and so on, as well as a more general variable count for graphs involving an arbitrary number of arcs.  These variable counts are polynomial in the problem input size and, when compared with the polynomial-time complexity of solving linear programs, supports Lemma~\ref{complexity_lemma} in the main paper.   We will define $N$ as the number of players and $m$ as the total number of nodes.

\hspace{1em}

\noindent\textbf{Grid:}  We first focus on the Grid family of networks, as used in Section~\ref{sec:experiments} in the main paper.  We note the following:
\begin{subequations}
\begin{equation}
    \text{Number of Origin-Destination Pairs: } (2) \binom{m}{2}
\end{equation}
\begin{equation}
    \text{Number of Arcs in a }\sqrt{m}x\sqrt{m}\text{ Grid: } (\sqrt{m} - 1)(\sqrt{m})(2)(2) 
\end{equation}
\end{subequations}

\noindent For the same costs across all players, we have to consider the variables $\text{diag}(\mathcal{C}),\bar{c},v_i^k,u_i^k,\bar{u}^k,\ \forall i, k$ along with the variables utilized in creating the 1-norms for the residual model.  This variable count comes out to be for our code:
\begin{equation}
    13Nm^3 - 12 Nm^{5/2} - 13 N m^2 + 12 N m^{3/2} + 16 m^3 - 16 m^{5/2} - 16 m^2 + 16 m^{3/2}
\end{equation}

\noindent For different costs across all players, we consider the variables $\text{diag}(\mathcal{C}_i),\bar{c}_i,v_i^k,u_i^k,\bar{u}^k,\ \forall i, k$ along with the variables utilized in creating the 1 norms for the residual model.  This variable count comes out to be for our code:
\begin{equation}
    21 N m^3 - 20 N m^{5/2} - 21 N m^2 + 20 N m^{3/2} + 8 m^3 - 8 m^{5/2} - 8 m^{2} + 8m^{3/2}
\end{equation}

\noindent Therefore, in big O notation, the order of both variable counts is $\mathcal{O}(N m^3)$.

\hspace{1em}

\noindent\textbf{General graphs:}  For any graph, we now define $a$ as the number of arcs, with $N$ as the number of players and $m$ as the total number of nodes.  Again, the number of OD pairs will be $(2) \binom{m}{2}$.  The variable count for the same costs across players is for our code:
\begin{equation}
    3aNm^2 - 3aNm + m^3 N - m^2 N + 4 am^2 - 4am
\end{equation}

\noindent The variable count for not the same costs across all players is for our code:

\begin{equation}
    5aNm^2 - 5aNm + m^3 N - m^2 N + 2am^2 - 2am
\end{equation}

\noindent Therefore, in big O notation, the order of both variable counts is $\mathcal{O}(aNm^2)$.


\subsection{Proof of Lemma \ref{thm:strong-monotonicity}} \label{appendix:strong-monotonicity-proof}

\begin{proof}
$\mathbf{F}$ is defined for (\ref{new_GNE}) as follows:
\begin{equation}
\mathbf{F}(\mathbf{x}) = 
\begin{bmatrix}
    2 \mathcal{C}_1 x_1 + \mathcal{C}_1 x_2 + ... + \mathcal{C}_1 x_N + \bar{c}_1 \\
    \mathcal{C}_2 x_1 + 2 \mathcal{C}_2 x_2 + ... + \mathcal{C}_2 x_N + \bar{c}_2\\
    \vdots \\
    \mathcal{C}_N x_1 + \mathcal{C}_N x_2 + ... + 2\mathcal{C}_N x_N + \bar{c}_N
\end{bmatrix}
\end{equation}

\noindent In order to prove strong monotonicity, we must show that there exists a scalar $\alpha > 0$ such that

\begin{equation}
    \left(\mathbf{F}(\mathbf{x}) - \mathbf{F}(\mathbf{y}) \right)^T (\mathbf{x}-\mathbf{y}) \geq \alpha ||\mathbf{x}-\mathbf{y}||^2\  \forall \mathbf{x},\mathbf{y} \in \mathbf{X},\ \mathbf{x} \neq \mathbf{y}.
\end{equation}

\noindent Through some algebra, it is not difficult to show that:
\begin{equation}
    \left(\mathbf{F}(\mathbf{x}) - \mathbf{F}(\mathbf{y}) \right)^T (\mathbf{x}-\mathbf{y}) = (\mathbf{x}-\mathbf{y})^T
    \begin{bmatrix}
        2 \mathcal{C}_1 & \mathcal{C}_1 & \hdots & \mathcal{C}_1 \\ \mathcal{C}_2 & 2 \mathcal{C}_2 & \hdots & \mathcal{C}_2 \\ \vdots & \vdots & \ddots & \vdots \\
        \mathcal{C}_N & \mathcal{C}_N & \hdots & 2 \mathcal{C}_N
    \end{bmatrix}^T (\mathbf{x}-\mathbf{y})
\end{equation}

\noindent If we assume that $\mathcal{C}_1 = \mathcal{C}_2 = ... = \mathcal{C}_N = \mathcal{C}$, we have 

\begin{equation}\label{equated_C_proof}
    = (\mathbf{x}-\mathbf{y})^T
    \begin{bmatrix}
        2 \mathcal{C} & \mathcal{C} & \hdots & \mathcal{C} \\ \mathcal{C} & 2 \mathcal{C} & \hdots & \mathcal{C} \\ \vdots & \vdots & \ddots & \vdots \\
        \mathcal{C} & \mathcal{C} & \hdots & 2 \mathcal{C}
    \end{bmatrix} (\mathbf{x}-\mathbf{y})
\end{equation}



\noindent From Proposition 2.2.10 in \citet{cottle_lcp}, we know that, for symmetric, real valued matrices $M$ and the smallest eigenvalue of $M$ referenced as $\lambda_1$, we can write $\lambda_1 ||x||^2 \leq x^T M x,\ \forall x$.  Therefore, to prove the matrix in (\ref{equated_C_proof}) is positive definite, we simply need to prove that its smallest eigenvalue is positive.  We begin this process by splitting the matrix in (\ref{equated_C_proof}) into two pieces
\begin{equation}
    A = \begin{bmatrix}
        \mathcal{C} & 0 & \hdots & 0 \\ 0 & \mathcal{C} & \hdots & 0 \\ \vdots & \vdots & \ddots & \vdots \\
        0 & 0 & \hdots & \mathcal{C}
    \end{bmatrix},\ B = \begin{bmatrix}
        \mathcal{C} & \mathcal{C} & \hdots & \mathcal{C} \\ \mathcal{C} & \mathcal{C} & \hdots & \mathcal{C} \\ \vdots & \vdots & \ddots & \vdots \\
        \mathcal{C} & \mathcal{C} & \hdots & \mathcal{C}
    \end{bmatrix}
\end{equation}

\noindent We state that $A+B = M$, with M as the original matrix, and we note that $A,B,M \in \mathbb{R}^{nN\times nN}$, with $n$ as the number of arcs and $N$ as the number of players.  First, we prove that the eigenvalues of $B$ are $\{ 0,...,0, N(eig(C))\}$.  For this proof, we notice that the eigenvalue-eigenvector equation for $B$ is 
\begin{equation}
    B z = \lambda z
\end{equation}
\begin{equation}\label{extended_eig_for_proof}
    \mathcal{C} z_1 + ... + \mathcal{C} z_N = \lambda z_i,\ i=1,...,N
\end{equation}

\noindent such that $z_i \in \mathbb{R}^n,\ \forall i$ (not all $z_i = 0$) and $\lambda \in \mathbb{R}$.  Because $\forall i$ the left-hand sides of (\ref{extended_eig_for_proof}) are the same, we can equate the right hand sides to say:
\begin{equation}
    \lambda z_1 = \lambda z_2 = ... = \lambda z_N
\end{equation}

\noindent This provides two cases for the eigenvalues.  Either $\lambda = 0$, or the $z_i$ are such that 
\begin{equation}
    z_1 = z_2 = ... = z_N = w \neq 0
\end{equation}

\noindent for some $w \in \mathbb{R}^n$.  This $w$ can be substituted back into the equation in (\ref{extended_eig_for_proof}) to say 
\begin{equation}
    N\  \mathcal{C} w = \lambda w
\end{equation}

\noindent which means the $\lambda$ are the eigenvalues of $N \mathcal{C}$ and, since we already know that $\mathcal{C}$ is a positive diagonal matrix, the eigenvalues are just those entries on the diagonal multiplied by $N$.  Because these diagonal values are positive and $N$ is positive, then the resultant eigenvalues are positive.  

Next, we use this information to prove that $M$ is a positive definite matrix. Since we know that $A$ and $B$ are real, symmetric matrices, Weyl's Inequality \cite[Theorem 4.3.1]{horn1985matrix} states:
\begin{equation}
    \lambda_1(A) + \lambda_1(B) \leq \lambda_1(A+B) = \lambda_1(M)
\end{equation}

\noindent We know from above that $\lambda_1(B) = 0$.  We also know that $\lambda_1(A) > 0$ because it is a diagonal matrix with the $\mathcal{C}$ matrices on its diagonal and, since $\mathcal{C}$ has all positive values, then we know all of the eigenvalues of $A$ are positive.  Consequently, $\lambda_1(A) + \lambda_1(B) > 0$ which means $\lambda_1(M) > 0$.  
As a result, we can write 
\begin{equation}
    (\mathbf{x}-\mathbf{y})^T
    \begin{bmatrix}
        2 \mathcal{C} & \mathcal{C} & \hdots & \mathcal{C} \\ \mathcal{C} & 2 \mathcal{C} & \hdots & \mathcal{C} \\ \vdots & \vdots & \ddots & \vdots \\
        \mathcal{C} & \mathcal{C} & \hdots & 2 \mathcal{C}
    \end{bmatrix} (\mathbf{x}-\mathbf{y}) \geq \lambda_1(M) ||\mathbf{x}-\mathbf{y}||^2
\end{equation}


\noindent and we know $\lambda_1(M)$ is a positive number.  Thus, we have proven strong monotonicity of $\mathbf{F}$.
\end{proof}

\subsection{Counterexample and MATLAB Code for Lemma \ref{lemma:not-all-C-equal}}\label{appendix_lemma:not-all-C-equal}

We will provide a counter example to demonstrate this lemma.  In MATLAB, if we set the \url{rng} seed to 1 with the \url{'twister'} option, we draw four arcs using the \url{rand} function and multiply these numbers by 1000.  We then form a $\mathcal{C}$ matrix by putting these four numbers along the diagonal.  Next, we form the following matrix:
\begin{equation}
    A = \begin{bmatrix}
    2\mathcal{C} & \mathcal{C} & \mathcal{C} & \mathcal{C} \\ (500.1)\mathcal{C} & 2(500.1)\mathcal{C} & (500.1)\mathcal{C} & (500.1)\mathcal{C} \\ (600.7)\mathcal{C} & (600.7)\mathcal{C} & 2(600.7)\mathcal{C} & (600.7)\mathcal{C} \\ (700.8)\mathcal{C} & (700.8)\mathcal{C} & (700.8)\mathcal{C} & 2(700.8)\mathcal{C}
    \end{bmatrix}
\end{equation}

\noindent Therefore, we see that we have formed a four player matrix game in which all of the players do have different costs due to the different factors multiplying the $\mathcal{C}$ matrices.  According to \cite{johnson1970positive,PD_wolfram}, a matrix with real values $A$ is positive definite if and only if $\frac{1}{2}(A+A^T)$ is positive definite.  We can check the smallest eigenvalue of this resultant matrix to find out if it is positive definite.  Using MATLAB's \url{eig} function, we discover that the smallest eigenvalue of the matrix $\frac{1}{2}(A+A^T)$ is -7.7307e+04.  Therefore, $A$ is not positive definite, thus demonstrating that, when the costs are not the same across all players, we are not guaranteed to have a strongly monotone $\mathbf{F}$ as defined by (\ref{specific_QVI_formulation}).  See below for MATLAB code.

This is the MATLAB code utilized for the counterexample for Lemma \ref{lemma:not-all-C-equal}:

\begin{verbatim}
rng(1,'twister'); 
v=1000*rand(1,4);
C=diag(v); 
M=[2*C  C   C   C
   500.1*C  2*500.1*C   500.1*C     500.1*C
   600.7*C  600.7*C  2*600.7*C      600.7*C
   700.8*C   700.8*C    700.8*C    2*700.8*C];
w=eig(0.5*(M+M'));
minval=min(w)
\end{verbatim}

\subsection{Additional Experimental Set-Up Details}\label{appendix:experimental_details}

The PATH solver \cite{dirkse1995path,ferris2000complementarity} in GAMS generates the data used in the inverse optimization residual model \citep{keshavarz2011imputing,ratliff2014social}.  The only default altered for the PATH solver was the tolerance threshold for the solver to finish, which changed from 1e-6 to 1e-8.  In the scripts, the PATH solver can begin at a few different starting points in case one of the start points fails due to solver error.  For the case in which the costs are the same across all players, $\mathbf{F}$ defined by (\ref{specific_QVI_formulation}) is strictly monotone, so the problem has a unique solution, which means the starting point does not matter.  For the case in which the costs are not the same across all players, multiple starting points are still provided to ensure that a solution is obtained to the problem.  Interestingly, all of the different cost matrices for our experiments are positive definite, except for one in grid size 5x5 and player number 10.\footnote{In different costs, we use slightly different starting point options between when we run the simulation for the original costs and when we run the simulation for the IO costs.  This doesn't appear to affect the results, as will be seen in the next few subsections.  We also did a check for $\alpha=10$, grid 5x5, and $N=10$ with the start points set to the same values for the original costs and IO costs, and the resulting errors were the same as before.}  Then, the \url{pyomo} \cite{hart2011pyomo,hart2017mathematical} package in Python is utilized to construct the inverse optimization residual model to find the parameterizations.  The Gurobi solver \cite{gurobi_citation} is utilized to solve the pyomo models, and the BarQCPConvTol and the BarConvTol parameters are set to 1e-14.  The documentation for BarConvTol states that smaller values for this parameter can lead to ``a more accurate solution'' \cite{barcovtol_gurobi}.  The experiments are run on machines with 4 cores and 32 GB of RAM.

The original costs under which the simulation data is generated have to be selected, and this was done randomly. MATLAB \cite{MATLAB:2020a} was utilized to generate the random sets of numbers needed, specifically the \url{unifrnd} random number function with 5 as the \url{rng} seed. This function draws from a uniform distribution the original costs for the $\mathcal{C}_i$ and $\bar{c}_i$ parameters in the simulation model outlined in (\ref{new_MCP}).  For the same costs across all players set-up, one $\mathcal{C}$ and one $\bar{c}$ are drawn but, when the costs vary across all the players, a $\mathcal{C}_i$ and a $\bar{c}_i$ are drawn for each player $i$.  For each experimental set up with a specified graph, number of players, and $\alpha$ level (the capacity level for the joint constraint), 10 trials are run.

\subsection{Positive Definiteness of Different Costs Matrices with MATLAB rng(5)}\label{appendix:different_cost_eigenvalues}

In Figure \ref{appendix-fig:minimum_eigenvalues_arc_76}, we plot a sampling of the minimum eigenvalues for the ``symmetric part'' \cite{PD_wolfram} of random matrices of the form:

\begin{equation}
    A = \begin{bmatrix}
        2 \mathcal{C}_1 & \mathcal{C}_1 & \hdots & \mathcal{C}_1 \\ \mathcal{C}_2 & 2 \mathcal{C}_2 & \hdots & \mathcal{C}_2 \\ \vdots & \vdots & \ddots & \vdots \\
        \mathcal{C}_N & \mathcal{C}_N & \hdots & 2 \mathcal{C}_N
    \end{bmatrix}
\end{equation}

\noindent in which $\mathcal{C}_i$ are diagonal matrices whose diagonals are chosen as random numbers from the uniform distribution from 1 to 5 using the MATLAB unifrnd function.  We choose to set the arc number to 76 (same number as the Sioux Falls arc number), and we gradually increase the number of players $N$ from 2 to 15.  We do use the seed of 5 for rng but, since we draw all of the numbers in one continuous session, the minimum eigenvalues here do not necessarily match the eigenvalues of the matrices that correspond with our experiments.  As discussed in Appendix \ref{appendix_lemma:not-all-C-equal} and according to \cite{johnson1970positive,PD_wolfram}, we can check the positive definiteness of the matrix $A$ by finding the minimum eigenvalue of $0.5(A+A^T)$, which is what we do for the 10 random matrices corresponding with each player setting.  In return, we obtain Figure \ref{appendix-fig:minimum_eigenvalues_arc_76}, and we see that, as player number increases, there is more and more of a chance that the minimum eigenvalue is negative, thus making the matrix not positive definite.

This is important because we can only guarantee unique solutions to our problem when the $A$ matrix is positive definite, as explained in Section \ref{sec:theory-QVI}.  This allows us to have multiple starting points in our code for the simulation experiments.

\begin{figure}[H]
\centering
\includegraphics[height=0.3\textheight,keepaspectratio]{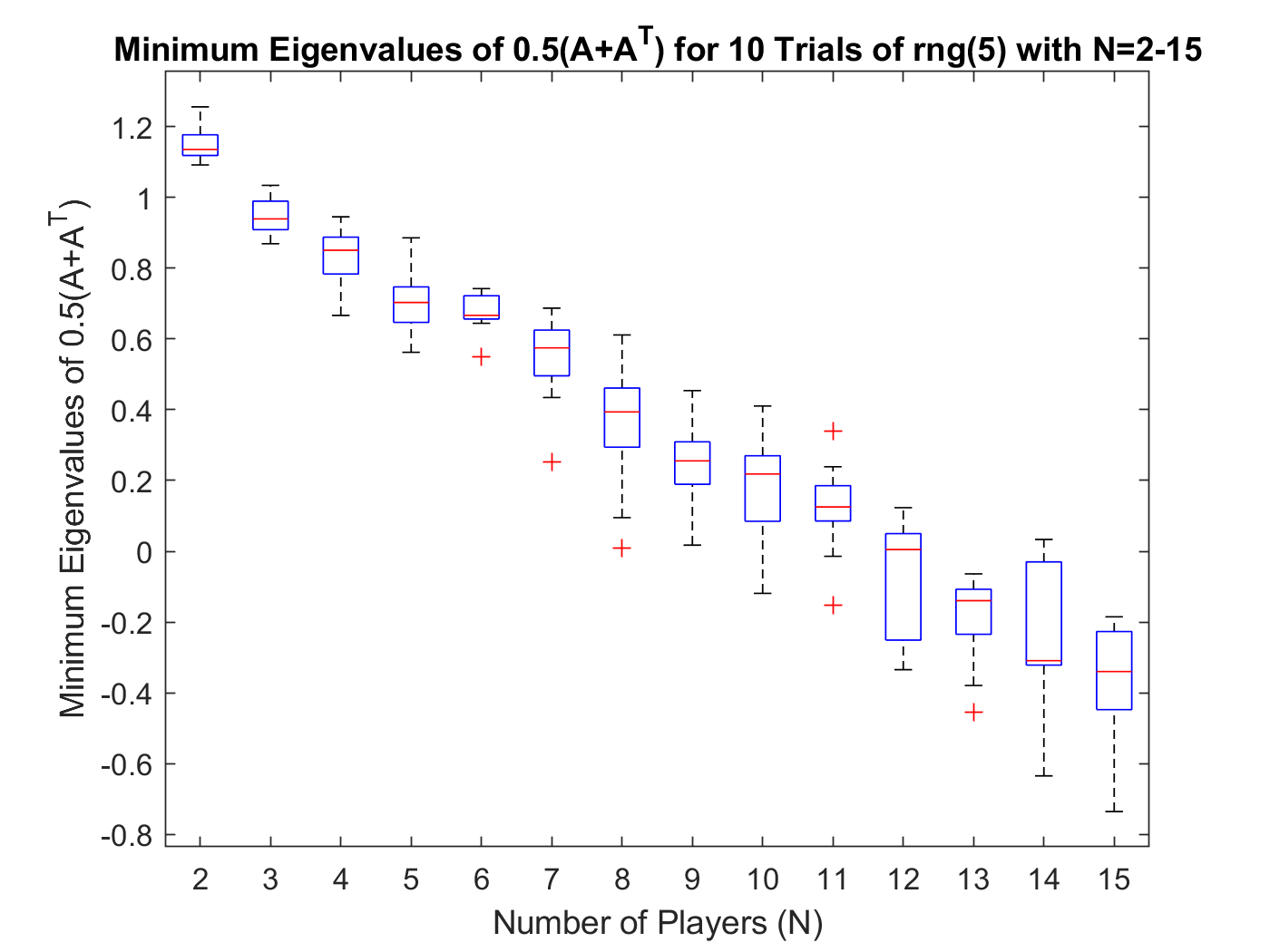}
\caption{Minimum Eigenvalues for 10 Trials, $N=2-15$, rng(5), and Arc Number=76, which is the same number of arcs as Sioux Falls}\label{appendix-fig:minimum_eigenvalues_arc_76}
\end{figure}

\subsection{Objective Function Values for the Experiment Groups}\label{appendix:obj_func_for_experiments}

Overall, it is important to remember that, for all the experiments, the objective function values are small, on the order of 1e-6 or lower.  However, it is also relevant to identify patterns where present.  For experiment group 1, in Figure \ref{fig:obj_func_experiment_1}, increasing grid size sometimes leads to a higher median objective function value in certain subsets of the experiments, including when there are 2 players and $\alpha=1$, when there are 5 players and $\alpha=2.5$, and when there are 10 players and $\alpha=5$.  These are all subsets in which alpha was cut in half according to the number of players present, which indicates that there could be a mild trend of increasing objective function values as grid size increases when the problem is more constrained.  Overall, however, the objective function values are small, with the highest outlier in terms of absolute value being a little over 2.5e-6 in the case of a 5x5 grid with 10 players and $\alpha=5$.\footnote{Note that there are a few negative values, but this isn't concerning because their corresponding flow errors are still low.}  For experiment group 2, in Figure \ref{fig:obj_func_experiment_2}, the objective values are overall quite low, with one outlier for 10 players and $\alpha=5$ above 8e-6.  They over all appear comparable to the grid objective function values once the scales of the two graphs are taken into account.\footnote{There do again appear to be some negative values but, again, this isn't concerning because their corresponding flow errors are still low.}  For experiment group 3, in Figure \ref{fig:obj_func_experiment_3}, although the objective function values of the inverse optimization residual model are on the order of 1e-6, the median values under the $\alpha = (0.5)N$ capacity are bigger than the $\alpha=N$ median values.  One theory for this trend is that the more constrained problems are more difficult to solve and hence lead to larger objective function values.  The ranges for these objective function values is about 1.5e-6 higher in the positive direction than the corresponding range for the values in Experiment group 1 (see Figure \ref{fig:obj_func_experiment_1}).  For experiment group 4, Figure \ref{fig:obj_func_experiment_4} demonstrates that most of the values fall below 1e-6, except for the $N=5, \alpha = 2.5$ experiment, whose maximum is close to 4e-6.  This does appear to mirror the the objective function values for experiment group 2 in terms of the general range, which also fell below 1e-6.

\begin{figure}[H]
\centering
\includegraphics[height=0.3\textheight,keepaspectratio]{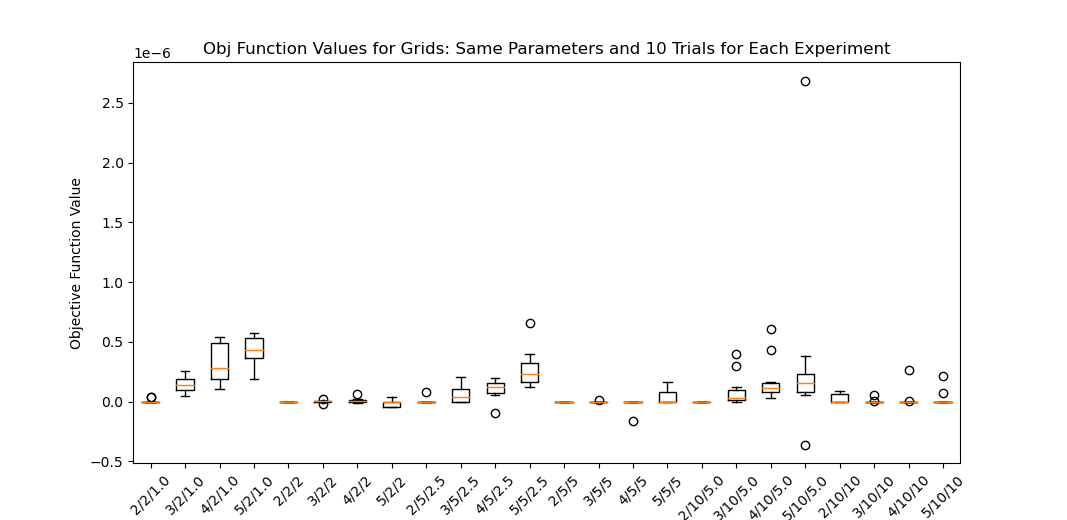}
\caption{Objective Function Values for Experiment Group 1: The labeling at the bottom of the graph indicates attributes of the boxplot, specifically the Grid Size/Number of Players/Alpha Value }\label{fig:obj_func_experiment_1}
\end{figure}

\begin{figure}[H]
\centering
\includegraphics[height=0.3\textheight,keepaspectratio]{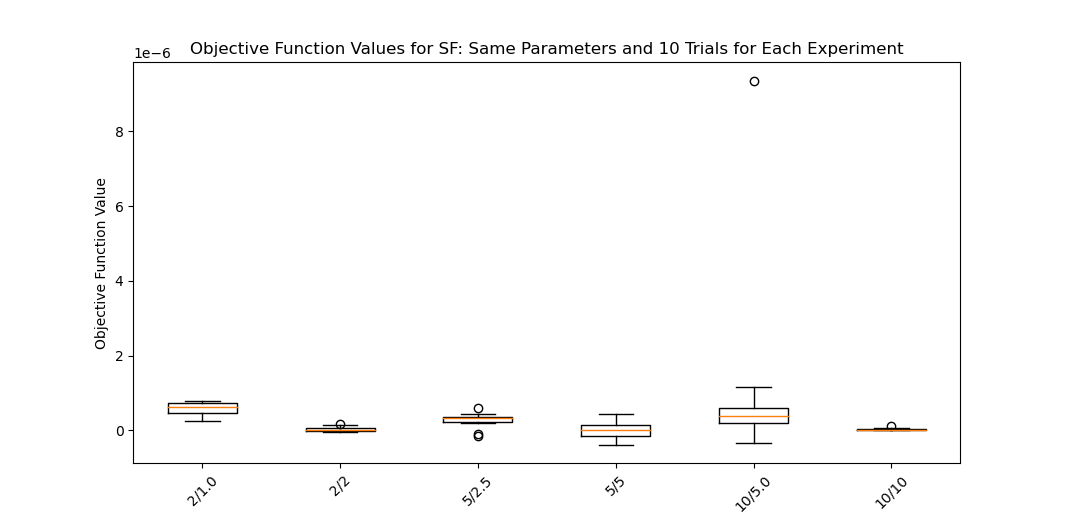}
\caption{Objective Function Values for Experiment Group 2: The labeling at the bottom of the graph indicates attributes of the boxplot, specifically the Number of Players/Alpha Value }\label{fig:obj_func_experiment_2}
\end{figure}

\begin{figure}[H]
\centering
\includegraphics[height=0.3\textheight,keepaspectratio]{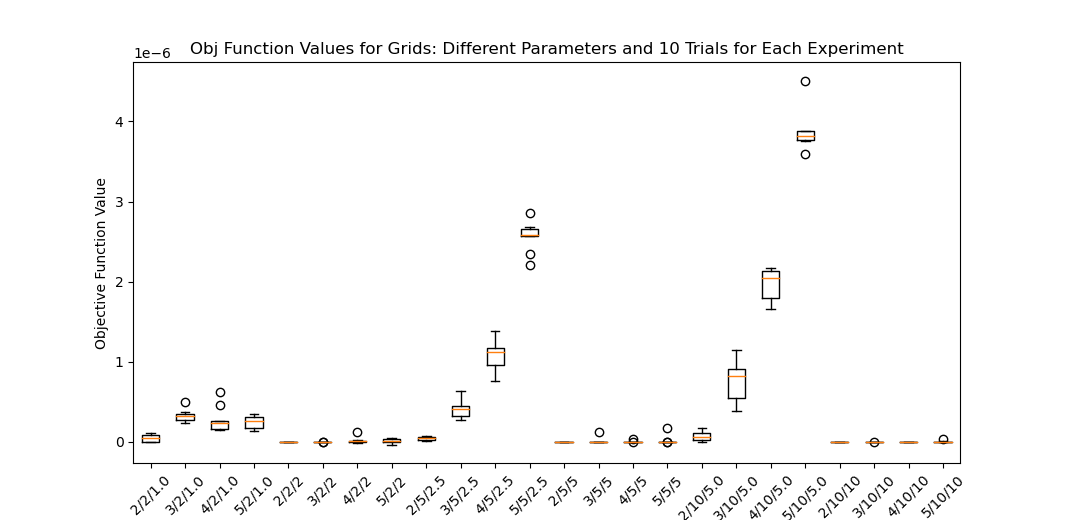}
\caption{Objective Function Values for Experiment Group 3: The labeling at the bottom of the graph indicates attributes of the boxplot, specifically the Number of Players/Alpha Value.  Note: Only 6 trials are included for 5/10/5.0, see note in the text.}\label{fig:obj_func_experiment_3}
\end{figure}

\begin{figure}[H]
\centering
\includegraphics[height=0.3\textheight,keepaspectratio]{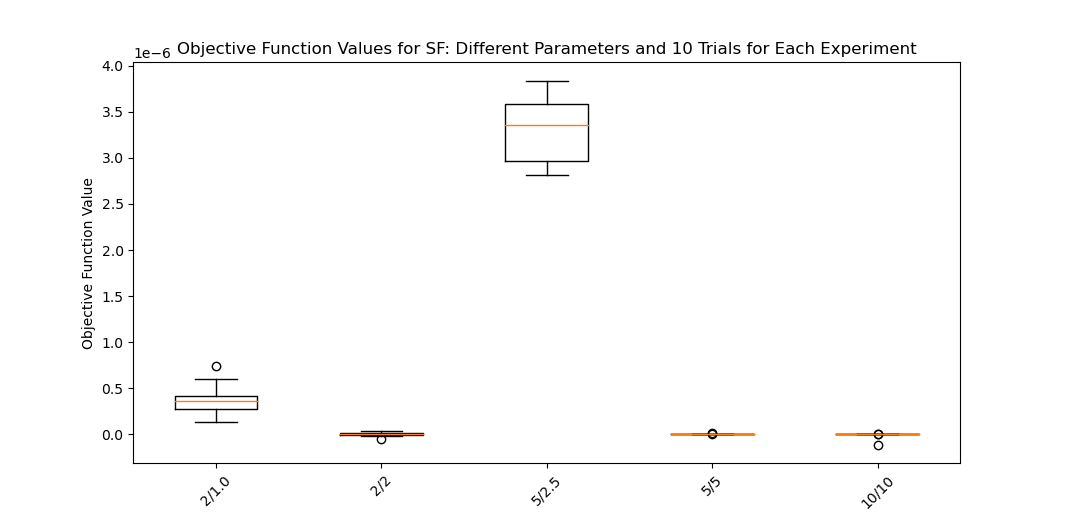}
\caption{Objective Function Values for Experiment Group 4: The labeling at the bottom of the graph indicates attributes of the boxplot, specifically the Number of Players/Alpha Value.  Note: The trials did not finish in time for $N=10$ and $\alpha=5$, hence that boxplot is not included.}\label{fig:obj_func_experiment_4}
\end{figure}

\subsection{Timing for the Experiment Groups}\label{appendix:timing_for_experiments}

In this subsection of the Appendix, we present the timing graphs for the experiment groups.  In Figure \ref{fig:timing_experiment_1}, timing increases with grid size for all of the subsets of $\alpha$ and number of players.  In Figure \ref{fig:timing_experiment_2}, the timing increases as the player number increases.  Both of these trends are understandable because of the effect on variable count that both grid size and player number have, as demonstrated by Appendix \ref{appendix:variable_number}.  In Figure \ref{fig:timing_experiment_3}, we see that all of the timing data is obscured by the large outlier of the experiment with grid 5x5, 10 players, and $\alpha=5$.  This experiment took quite a while and, as previously noted, only 6 of the trials for the experiment finished in under 24 hours (the 1st, 2nd, 3rd, 4th, 5th, and 7th).  We hypothesize that having the largest grid, the most number of players, the more constrained $\alpha$ value of the two, and the more difficult task of determining different sets of cost parameters for each player caused the large time amounts. In Figure \ref{fig:timing_experiment_4}, again noting that the trials did not finish in time for $N=10$ and $\alpha=5$, it appears that there are not trends to discuss, although note that all the experiments but $N=5, \alpha=2.5$ have consistent timing data across trials. 

\begin{figure}[H]
\centering
\includegraphics[height=0.3\textheight,keepaspectratio]{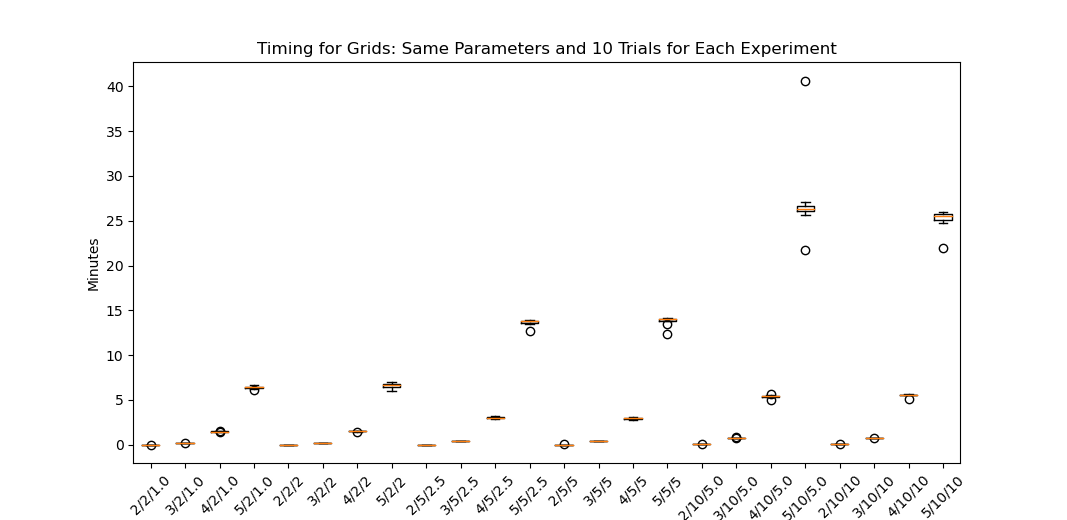}
\caption{Timing for Experiment Group 1: The labeling at the bottom of the graph indicates attributes of the boxplot, specifically the Grid Size/Number of Players/Alpha Value }\label{fig:timing_experiment_1}
\end{figure}

\begin{figure}[H]
\centering
\includegraphics[height=0.3\textheight,keepaspectratio]{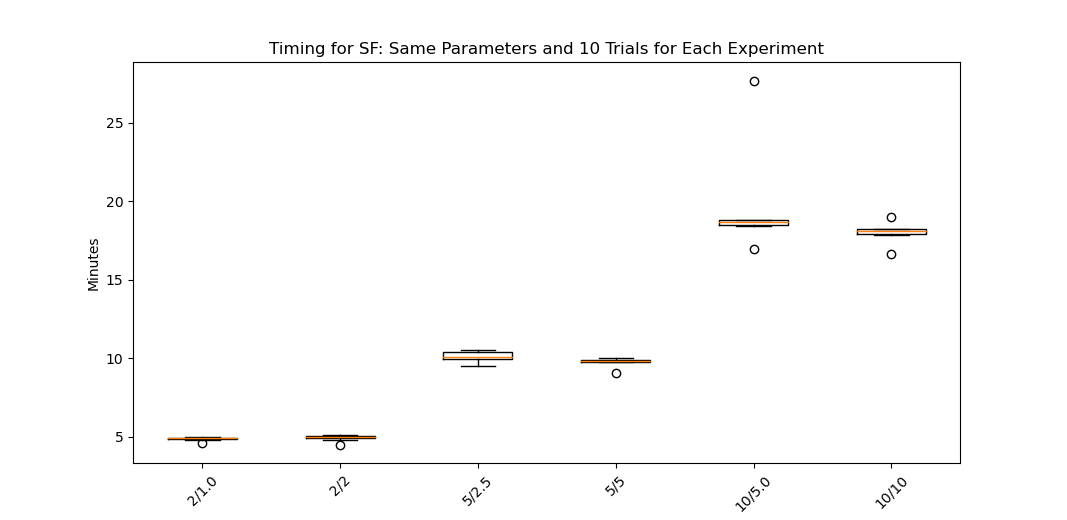}
\caption{Timing for Experiment Group 2: The labeling at the bottom of the graph indicates attributes of the boxplot, specifically the Number of Players/Alpha Value }\label{fig:timing_experiment_2}
\end{figure}

\begin{figure}[H]
\centering
\includegraphics[height=0.3\textheight,keepaspectratio]{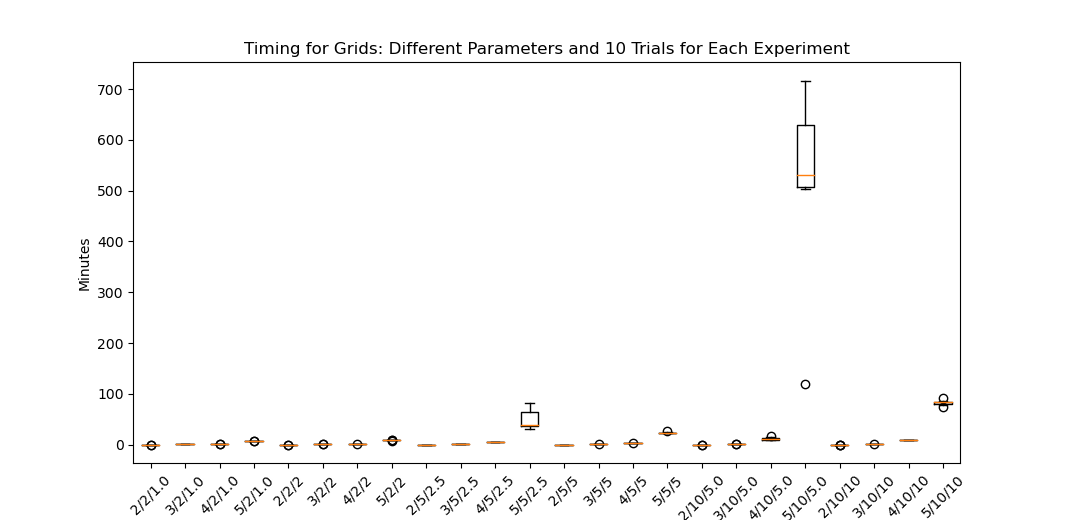}
\caption{Timing for Experiment Group 3: The labeling at the bottom of the graph indicates attributes of the boxplot, specifically the Grid Size/Number of Players/Alpha Value. Note: Only 6 trials are included for 5/10/5.0, see note in the text.}\label{fig:timing_experiment_3}
\end{figure}

\begin{figure}[H]
\centering
\includegraphics[height=0.3\textheight,keepaspectratio]{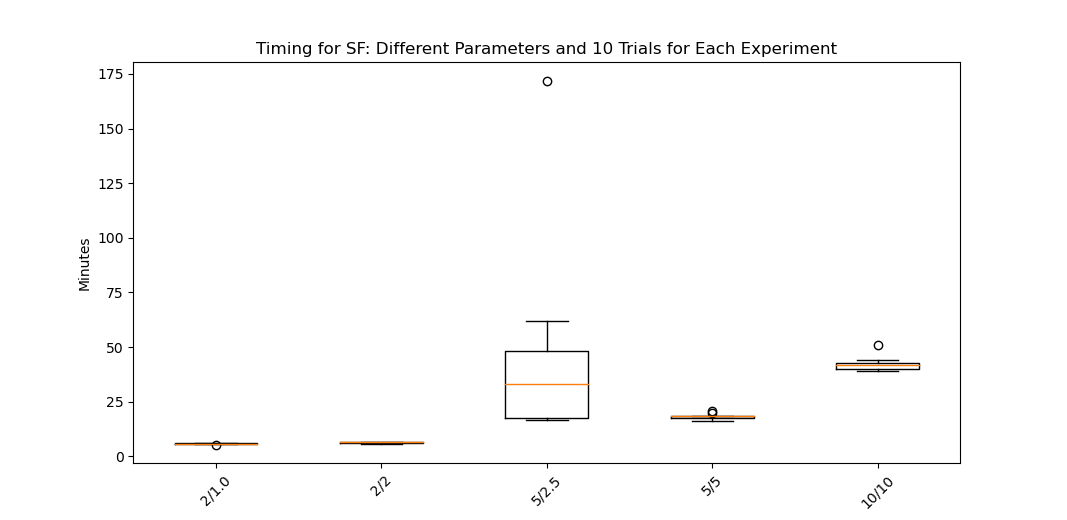}
\caption{Timing for Experiment Group 4: The labeling at the bottom of the graph indicates attributes of the boxplot, specifically the Number of Players/Alpha Value. Note: The trials did not finish in time for $N=10$ and $\alpha=5$, hence that boxplot is not included.
}\label{fig:timing_experiment_4}
\end{figure}

\section{Code Attribution}

In this section, we would like to cite the various packages and software that we used in the project.  We would also like to note that the code for this project was built off of another code base that is for another on-going project (not published yet).

\begin{itemize}
    \item Python Packages:
    \begin{itemize}
        \item \url{pyomo} \citep{hart2011pyomo,hart2017mathematical}
        \item \url{networkx} \citep{networkx_citation}
        
        \item \url{matplotlib} \citep{Hunter:2007} 
        
        \item \url{numpy} \citep{numpy_citation,numpy_citation_2,2020SciPy-NMeth} 
        
        \item \url{scipy} \citep{2020SciPy-NMeth}
        
        \item \url{pandas} \citep{mckinney2010data}
    \end{itemize}
    
    \item MATLAB 9.8.0.1323502 (R2020a) \citep{MATLAB:2020a}
    
    \item GAMS \citep{GAMS_software_33_2,GAMS_software_34_1} with PATH solver \citep{dirkse1995path,ferris2000complementarity} (PATH website \citep{path_website})
    
    \item Solvers:
    \begin{itemize}
        \item \url{gurobi} Version 9.1.1 \citep{gurobi_citation}; Gurobi Documentation \citep{gurobi_documentation}
        
        
    \end{itemize}
    
    \item Data: Sioux Falls data \citep{sioux_falls_data,leblanc1975efficient}
    
\end{itemize}

These are some of the helpful resources we utilized for writing our code: 

\begin{itemize}
\item These articles \citep{birge1992efficient,lustig1991formulating} were helpful in determining some of the constraints that equalized necessary variables in my code.

\item We learned much about simulation from \citet{dong2018generalized} and \citet{barmann2018online,barmann2017emulating}, specifically about running multiple trials and the importance of randomization.

\end{itemize}




\end{document}